\documentclass[a4paper,11pt, twoside]{amsart}
\usepackage[inner=2cm,outer=2cm,top=2.5cm,headsep=1cm, footskip=1cm,bottom=3cm]{geometry}
\usepackage[english]{babel}
\usepackage{amsfonts,latexsym,rawfonts,amsmath,amssymb,amsthm,amscd}
\usepackage{enumerate}
\usepackage{stackrel}
\usepackage{hyperref}
\usepackage{graphicx}
\usepackage{mdwlist}
\usepackage{array} 
\input xy
\xyoption{all}
\usepackage{color}
\usepackage[table]{xcolor}
\usepackage{tikz}

\hypersetup{colorlinks, citecolor=black, filecolor=black,linkcolor=black,urlcolor=black} 

\newtheorem{prop}{Proposition}[section]
\newtheorem{teo}[prop]{Theorem}
\newtheorem{lem}[prop]{Lemma}
\newtheorem{cor}[prop]{Corollary}

\theoremstyle{definition}
\newtheorem{defi}[prop]{Definition}
\newtheorem{example}[prop]{Example}
\newtheorem{rmk}[prop]{Remark}
\renewenvironment{itemize}{
  \begin{list}{}{
    \setlength{\leftmargin}{2em}
    \setlength{\itemsep}{0.25em}
    \setlength{\parskip}{0pt}
    \setlength{\parsep}{0.25em}
  }
}{
  \end{list}
}
\newcommand{\jb}{\begin{picture}(-1,1)(1,-3)\circle*{2.5}\end{picture}}

\newcommand{\Ho}{\mathrm{Ho}}
\newcommand{\Hom}{\mathrm{Hom}}
\newcommand{\Ker}{\mathrm{Ker }}
\newcommand{\Coker}{\mathrm{Coker }}
\newcommand{\Img}{\mathrm{Im }}
\newcommand{\Dec}{\mathrm{Dec}}

\newcommand{\pb}{\ar@{}[dr]|{\mbox{\LARGE{$\lrcorner$}}}}
\newcommand{\xra}[1]{\xrightarrow{#1}}
\newcommand{\lra}{\longrightarrow}
\newcommand{\du}{\vee}
\newcommand{\wt}{\widetilde}
\newcommand{\wh}{\widehat}
\newcommand{\ov}[1]{\overline{#1}}

\newcommand{\gris}{\cellcolor{black!8}}

\newcommand{\pdga}[1]{\widehat{\mathcal{P}}\mathsf{CDGA}_{#1}}
\newcommand{\pga}[1]{\widehat{\mathcal{P}}\mathsf{CGA}_{#1}}
\newcommand{\GMpdga}[1]{{\mathcal{P}}\mathsf{CDGA}_{#1}}

\newcommand{\CC}{\mathbb{C}}

\newcommand{\PP}{\mathbb{P}}
\newcommand{\QQ}{\mathbb{Q}}

\newcommand{\ZZ}{\mathbb{Z}}

\newcommand{\Aa}{\mathcal{A}}

\newcommand{\Ii}{\mathcal{I}}
\newcommand{\Jj}{\mathcal{J}}

\newcommand{\Mm}{\mathcal{M}}

\newcommand{\Pp}{\mathcal{P}}

\newcommand{\Vv}{\mathcal{V}}

\newcommand{\Xx}{\mathcal{X}}

\newcommand{\kk}{\mathbf{k}}

\title
[MHS on the intersection homotopy type of varieties with isolated singularities]
{Mixed Hodge structures on the intersection homotopy type of complex varieties with isolated singularities}

\author{David Chataur}
\address[D. Chataur]{Laboratoire Ami\'{e}nois de Math\'{e}matique Fondamentale et Appliqu\'{e}e\\  Universit\'{e} de Picardie Jules Verne\\ 
33, rue Saint-Leu 80039 Amiens Cedex 1, France}
\email{david.chataur@u-picardie.fr}

\author{Joana Cirici}
\address[J. Cirici]{
Fachbereich Mathematik und Informatik\\
Freie Universit\"{a}t Berlin\\  Arnimallee 3\\ 
14195 Berlin, Germany}
\email{jcirici@math.fu-berlin.de}

\thanks{The second-named author would like to acknowledge financial support from
the German Research Foundation (SPP-1786) and partial support from
the Spanish Ministry of Economy and Competitiveness (MTM2013-42178-P)}


\begin{document}
\maketitle

\begin{abstract}
A homotopical treatment of intersection 
cohomology recently developed by Chataur-Saralegui-Tanr\'{e} 
associates a \textit{perverse algebraic model} to every topological pseudomanifold,
extending Sullivan's presentation of rational homotopy to intersection cohomology.
In this context, there is a notion of \textit{intersection-formality}, 
measuring the vanishing of Massey products in intersection cohomology.
In the present paper, we study the perverse algebraic model of 
complex projective varieties with isolated singularities.
We endow such invariant with natural mixed Hodge structures.
This allows us to prove some intersection-formality results
for large families of complex projective varieties, 
such as isolated surface singularities and varieties of arbitrary dimension with
ordinary isolated singularities.
\end{abstract}

\section{Introduction}
The intersection cohomology 
of a complex projective variety
enjoys many of the good properties of the
ordinary cohomology of a smooth variety,
collectively known as the \textit{K\"{a}hler package} 
(Poincar\'{e} duality, weak and hard Lefschetz,
Hodge decomposition and Hodge signature theorem).
These properties deal primarily with the
intersection cohomology group 
that has attracted most of the attention from algebraic topologists and geometers:
the middle-perversity intersection cohomology group.
However, there is additional geometric information carried by other
intersection cohomology groups, as well as
by cohomological operations that are defined when allowing other perversities than the middle one
(such as cup products or Steenrod squares).
It is in this context, that Goresky raised the following question in the introduction of \cite{Goresky}:
\begin{quotation}
{\small ``It remains as open question whether there is an intersection 
homology-analogue
to the rational homotopy theory of Sullivan. For example,
one would like to know when Massey triple products are defined
in intersection homology
and whether they always vanish on a (singular) projective algebraic variety''.}
\end{quotation}
The first part of Goresky's question has been answered by Chataur-Saralegi-Tanr\'{e}
in the foundational work \cite{CST} on rational intersection homotopy theory,
where the \textit{perverse algebraic model} 
of a topological pseudomanifold is introduced. This is 
a perverse commutative differential graded algebra (perverse cdga for short) defined over the rationals,
whose cohomology is isomorphic to the rational
intersection cohomology with all perversities
and is, when forgetting its multiplicative structure, quasi-isomorphic to the intersection cochains
originally introduced by
Goresky and MacPherson \cite{GMP1,GMP2}. In general, the perverse algebraic model 
contains more information than the intersection cohomology ring (for instance, it contains the Massey products) and
gives rise to a well-defined notion of intersection-formality for 
topological pseudomanifolds, analogously to the notion of formality 
appearing in the classical rational homotopy theory of Sullivan \cite{Su}.

Other significant contributions in this direction are the homotopy theory of perverse cdga's
developed by Hovey \cite{Hov2} within the context of Quillen model categories,
the works of Friedman \cite{Friedman} and Friedman-McClure \cite{FMC} 
on intersection pairings and cup products in intersection cohomology respectively
and Banagl's theory of intersection spaces \cite{Banagl}.

The present work draws its main motivation from the second
part of Goresky's question, which is almost equivalent to asking whether
singular complex projective varieties are intersection-formal.
This question is legitimated by a well-known application of the Hodge decomposition to topology: the Formality Theorem 
of Deligne-Griffiths-Morgan-Sullivan \cite{DGMS}, which states that
the rational homotopy type of a compact K\"{a}hler manifold
is entirely determined by its cohomology ring.

In general, the Hodge decomposition on the intersection cohomology of a singular projective variety
fails for perversities other than the middle one. Instead, 
each intersection cohomology group carries a mixed Hodge structure. Since the perverse algebraic model
depends on all perversities, we do not
expect an intersection-analogous statement of the Formality Theorem for singular projective varieties,
but of a generalization of this statement involving the weight spectral sequence.

In this paper, we study the rational intersection homotopy type
of complex projective varieties with only isolated singularities,
via mixed Hodge theory.
\\

We next explain the contents and main results of this paper. For the rest of this introduction, let $X$ be a complex
projective variety with only isolated singularities.

In Section $\ref{Section_RIHT}$, we
collect preliminary definitions and results
on intersection cohomology and on the homotopy theory of perverse cdga's.
Following \cite{CST}, we
describe the perverse algebraic 
model $I\Aa_{\ov\bullet}(X)$ of $X$.
This can be computed from the morphism
of rational algebras of piece-wise linear forms $\Aa_{pl}(X_{reg})\to \Aa_{pl}(L)$
induced by the inclusion $L\hookrightarrow X_{reg}$
of the link $L$ of the singularities into the regular part of $X$.

Section $\ref{Section_MHS}$ is the core of this paper. 
In this section, we endow the perverse algebraic model $I\Aa_{\ov\bullet}(X)$
of $X$ with natural mixed Hodge structures 
(this result is stated in a more technical form in Theorem $\ref{MHSmodel}$). Our proof relies, first, on the existence of 
mixed Hodge structures on the rational homotopy types of $X_{reg}$ and $L$ due to Morgan \cite{Mo} and 
Durfee-Hain \cite{DH} respectively, and second, on the existence of relative models of mixed Hodge diagrams
proven by Cirici-Guill\'{e}n in \cite{CG1}. 
As an important application of the existence of mixed Hodge structures on the perverse algebraic model,
we study the \textit{perverse weight spectral sequence} 
$IE_{1,\ov \bullet}^{*,*}(X)$,
a perverse differential bigraded algebra
whose cohomology computes the weight filtration on the intersection cohomology:
$IE_{2,\ov \bullet}^{*,*}(X):=H^{*,*}(IE_{1,\ov \bullet}(X))\cong Gr_{\bullet}^WIH^*_{\ov\bullet}(X;\QQ)$.
In Theorem $\ref{IE1formality}$, we prove that 
the complex intersection homotopy type of $X$ is a direct consequence of 
its perverse weight spectral sequence. In other words: there is a string of quasi-isomorphisms
of perverse cdga's from
$I\Aa_{\ov\bullet}(X)\otimes\CC$ to 
$IE_{1,\ov \bullet}(X)\otimes\CC$. This result descends to the rationals for perverse cdga's of finite type 
and is the intersection-analogue of the main result of \cite{CG1},
which in turn is the generalization to singular varieties,
of the Formality Theorem of \cite{DGMS}. 
As in the classical setting,
the perverse weight spectral sequence can be described 
in terms of the cohomologies of varieties associated with a resolution of singularities of $X$.
Hence Theorem $\ref{IE1formality}$ implies that the complex intersection homotopy type of
$X$ has a finite-dimensional model, determined by
cohomologies of smooth projective varieties.

The last two sections contain applications of Theorem $\ref{IE1formality}$.
In Section $\ref{Section_OIS}$, we prove
that if $X$
admits a resolution of singularities 
in such a way that the exceptional divisor is smooth, and if the link
of each singular point is $(n-2)$-connected, where $n$ is the complex dimension of $X$, then $X$ is GM-intersection-formal
over $\CC$ (the prefix GM accounts for Goresky-MacPherson,
since we consider finite perversities only).
The main class of examples to which this result applies are varieties with 
ordinary multiple points, but it also applies to a large family of
hypersurfaces and more generally, to complete
intersections
admitting  a resolution of singularities with smooth exceptional divisor. 
This extends a result of \cite{CST}, where it is shown that any
nodal hypersurface of $\CC\PP^4$ is intersection-formal.
Likewise, in Section $\ref{Section_OIS}$ we prove GM-intersection-formality over $\CC$
for every isolated surface singularity.
If a variety is (GM)-intersection-formal,
then its normalization is formal in the classical sense.
We remark that these results generalize our previous work
\cite{ChCi1}, where we study the (classical) rational homotopy type of complex 
projective varieties with normal isolated singularities,
using the multiplicative weight spectral sequence.

\section{Rational intersection homotopy types}\label{Section_RIHT}
In this preliminary section, we recall the description 
of the intersection cohomology of a complex projective variety with only isolated singularities
appearing in \cite{GMP1}, as well as its main properties. Then, we
introduce the notion of rational intersection homotopy equivalence and study its relation with
the classical notion of rational homotopy equivalence.
Lastly, we collect the necessary definitions and results
on the homotopy theory of perverse differential graded algebras,
such as the intersection-analogous notions of quasi-isomorphism and formality,
and describe the perverse algebraic model of a complex
projective variety with only isolated singularities, following \cite{CST}.

\subsection{Intersection cohomology}\label{sectionintersectioncohomology}
Intersection cohomology is defined for any topological pseudomanifold and
depends on the choice of a multi-index called \textit{perversity},
which measures how far cycles are allowed to deviate from transversality.
For a complex projective variety of dimension $n$ having only isolated singularities,
a perversity $\ov p$ is determined by a single integer $p$ such that
$0\leq p\leq 2n-2$. We will denote by $\Pp$ the totally ordered set of such perversities.
There are three distinguished elements in $\Pp$ that we shall refer to:
the $\ov{0}$-perversity $\ov{0}=0$, the middle perversity $\ov{m}=n-1$
and the top perversity $\ov{t}=2n-2$. The complementary perversity of $\ov p\in\Pp$ is given by $\ov{t}-\ov{p}:=2n-2-p$.
Note that the middle perversity is complementary to itself.
We enlarge the set of perversities $\widehat{\Pp}=\Pp\cup \{\ov{\infty}\}$ 
by adjoining the $\ov{\infty}$-perversity. We define
the sum of two perversities $\ov p$ and $\ov q$ in $\widehat{\Pp}$ by letting
$\ov p+\ov q:=\ov{p+q}$ if $p+q\leq 2n-2$ and $\ov p+\ov q:=\ov\infty$ otherwise.

Let $X$ be a complex projective variety of dimension $n$ with only isolated singularities. Denote by
$\Sigma$ the singular locus of $X$ and by $X_{reg}:=X-\Sigma$ its regular part.
The intersection cohomology of $X$ with perversity $\ov p$ (and coefficients in a commutative ring $R$) is given by
(see $\S$6.1 of \cite{GMP1})
$$
IH^k_{\ov{p}}(X;R)=\left\{
\begin{array}{ll}
H^k(X_{reg};R)&\text{ if }k\leq p\\
\Img\left(H^k(X;R)\lra H^k(X_{reg};R)\right)&\text{ if }k=p+1\\
H^k(X;R)&\text{ if }k>p+1
\end{array}
\right..
$$

For the $\ov{0}$-perversity we have an isomorphism of graded algebras $IH^*_{\ov0}(X;R)\cong H^*(\overline{X};R)$, 
where $\overline{X}\to X$ is a normalization of $X$
(see $\S$4 of \cite{GMP1}). 
For the $\ov\infty$-perversity we recover the cohomology ring $IH^*_{\ov{\infty}}(X;R)\cong H^*({X}_{reg};R)$ 
of the regular part of $X$ (see \cite{CST}).
A main feature of intersection cohomology is that, when $R=\QQ$,
for every finite perversity $\ov p\in \Pp$ we have a Poincar\'{e} duality isomorphism (see $\S$3.3 of \cite{GMP1})
$$IH^k_{\ov{p}}(X;\QQ)\cong (IH^{2n-k}_{\ov{t}-\ov{p}}(X;\QQ))^\du:=\Hom(IH^{2n-k}_{\ov{t}-\ov{p}}(X;\QQ),\QQ).$$

The graded objects $IH_{\ov p}^*(X;R)$ together with the morphisms
$IH_{\ov p}^*(X;R)\lra IH_{\ov q}^*(X;R)$ for every pair $\ov p\leq \ov q$,
and the products
$IH_{\ov p}(X;R)\otimes IH_{\ov q}(X;R)\lra IH_{\ov p+\ov q}(X;R)$ 
induced by the cup products of $H^*(X;R)$ and $H^*(X_{reg};R)$
for any pair $\ov p,\ov q\in \widehat \Pp$,
constitute the prototypical example of a \textit{perverse commutative graded $R$-algebra}:
this is a commutative monoid in the category
of functors from $\widehat\Pp$ to the category of graded $R$-modules.

Denote by $\Vv_\CC$ the category whose objects are complex projective varieties with only isolated singularities 
and whose morphisms $f:X\lra Y$ satisfy $f(X_{reg})\subset Y_{reg}$.
The above formula defines a contravariant functor $IH_{\ov\bullet}^*(-;R):\Vv_\CC\lra \pga{R}$
with values in the category of perverse commutative graded $R$-algebras.

\subsection{Intersection homotopy equivalence}
The consideration of the intersection cohomology ring with all perversities
leads to a natural notion of rational intersection homotopy equivalence.

\begin{defi}Let
$f:X\lra Y$ be a morphism between simply connected topological pseudomanifolds, such that $f(X_{reg})\subset Y_{reg}$.
Then $f$ is said to be a \textit{rational intersection homotopy equivalence}
if it induces an isomorphism of perverse graded algebras $f^*:IH^*_{\ov \bullet}(Y;\QQ)\lra IH^*_{\ov \bullet}(X;\QQ)$.
\end{defi}

If $f:X\to Y$ is a rational intersection homotopy equivalence then the morphism induced on
the normalizations $\overline f:\overline X\lra \overline Y$ is a rational
homotopy equivalence. The following result exhibits how the notion of rational
intersection homotopy equivalence is stronger than the classical notion of rational homotopy equivalence.

\begin{prop}\label{proj_cones_htp_equiv}
Let $S$ and $S'$ be two simply connected smooth projective surfaces of $\CC\PP^n$. 
Denote by $\PP_cS$ and $\PP_cS'$ the projective cones of $S$ and $S'$ respectively.
Then:
\begin{enumerate}[(1)]
 \item   $\PP_cS$ and $\PP_cS'$ are rationally homotopy equivalent if and only if $\Xx(S)=\Xx(S')$.
 \item  $\PP_cS$ and $\PP_cS'$ are rationally intersection homotopy equivalent if and only if $S$ and $S'$ 
 are rationally homotopy equivalent.
\end{enumerate}
\end{prop}
\begin{proof}
Let $w\in H^2(S;\QQ)$ denote the Poincar\'{e} dual of the
hyperplane section of $S\subset\CC\PP^n$. Since $w^2\neq 0$, using
Poincar\'{e} duality we obtain an orthogonal decomposition 
$H^2(S;\QQ)\cong\QQ\langle w\rangle \oplus^\bot V.$
The projective cone $\PP_cS$ of $S$ is isomorphic to
the Thom space of the restriction $S(1)$ of the hyperplane bundle on $\CC\PP^n$ to $S$.
The rational cohomology algebra of $\PP_cS$ 
can be written as $H^*(\PP_cS;\QQ)\cong \QQ\langle Th\rangle \oplus V'$,
where $Th$ has degree $2$ and satisfies $Th^4=0$ and $V'$ is a vector space of degree 4.
Thom's isomorphism $\cup Th: H^*(S;\QQ)\to \wt H^*(\PP_cS;\QQ)$ identifies $w$ with $Th^2$
and $V$ with $V'$. Furthermore, $Th\cup V'=0$. This proves (1).
The intersection cohomology of $\PP_cS$ can be written as:
\begin{equation*}
IH^s_{\ov p}(\PP_cS;\QQ)\cong
\def\arraystretch{1.4}
\begin{array}{| c | c | c | }
\multicolumn{1}{c}{\text{\tiny{$\ov p=\ov 0$}}}&\multicolumn{1}{c}{\text{\tiny{$\ov p=\ov m$}}}&\multicolumn{1}{c}{\text{\tiny{$\ov p=\ov t$}}}\\
\hline
\QQ\langle Th^3\rangle   &\QQ\langle Th^3\rangle&\QQ\langle Th^3\rangle  \\ \hline
0&0&0\\ \hline
\QQ\langle Th^2\rangle\oplus V'  &\QQ\langle Th^2\rangle\oplus V'  & H^4(S;\QQ)   \\ \hline
0&0&0  \\ \hline
\QQ\langle Th\rangle  &\QQ\langle w\rangle \oplus V&   \QQ\langle w\rangle \oplus V                   \\\hline
0  &0&0                     \\\hline
\QQ  &\QQ&\QQ                         \\\hline
\end{array}
\def\arraystretch{1.4}
\begin{array}{ l }
\\
\text{\tiny{$s=6$}}\\
\text{\tiny{$s=5$}}\\
\text{\tiny{$s=4$}}\\
\text{\tiny{$s=3$}}\\
\text{\tiny{$s=2$}}\\
\text{\tiny{$s=1$}}\\
\text{\tiny{$s=0$}}
\end{array}
\end{equation*}

where the product $IH_{\ov m}^2(\PP_cS;\QQ)\otimes IH_{\ov m}^2(\PP_cS;\QQ)\lra IH_{\ov t}^4(\PP_cS;\QQ)\cong H^4(S;\QQ)=\QQ$
corresponds to the product on $H^2(S;\QQ)$ and determines the signature of $S$.
This proves (2).
\end{proof}

\begin{example}
Let $S$ be a K3-surface and let $S'$ be the projective plane blown-up at 19 
points. Then $\Xx(S)=\Xx(S')=24$, $Sign(S)=(3,19)$ and $Sign(S')=(1,21)$.
Therefore $\PP_cS$ and $\PP_cS'$ are rationally homotopy equivalent, but not rationally 
intersection homotopy equivalent.
\end{example}

\subsection{Integral intersection cohomology}
We prove an analogous statement of Proposition $\ref{proj_cones_htp_equiv}$
for intersection cohomology with integer coefficients.

\begin{prop}\label{proj_cone_homeo}
Let $S$ and $S'$ be two simply connected smooth projective surfaces of $\CC\PP^n$. 
Then their projective cones $\PP_cS$ and $\PP_cS'$ are homeomorphic if and only if
$IH^*_{\ov \bullet}(\PP_cS;\ZZ)$ and $IH^*_{\ov \bullet}(\PP_cS';\ZZ)$ are isomorphic as perverse graded algebras.
\end{prop}
\begin{proof}
We follow the notation of the proof of Proposition $\ref{proj_cones_htp_equiv}$.
The intersection cohomology algebra of $\PP_cS$ is given by:
\begin{equation*}
IH^s_{\ov p}(\PP_cS;\ZZ)\cong
\def\arraystretch{1.4}
\begin{array}{| c | c | c | }
\multicolumn{1}{c}{\text{\tiny{$\ov p=\ov 0$}}}&\multicolumn{1}{c}{\text{\tiny{$\ov p=\ov m$}}}&\multicolumn{1}{c}{\text{\tiny{$\ov p=\ov t$}}}\\
\hline
\ZZ\langle T\rangle , Th^3=\deg(S)\cdot T&\ZZ\langle T\rangle&\ZZ\langle T\rangle  \\ \hline
0&0&0\\ \hline
H^4(\PP_cS;\ZZ)\cong H^2(S;\ZZ) &H^4(\PP_cS;\ZZ)\cong H^2(S;\ZZ) & H^4(S;\ZZ)\cong \ZZ   \\ \hline
0&0&0  \\ \hline
\ZZ\langle Th\rangle  &H^2(S;\ZZ)&     H^2(S;\ZZ)    \\\hline
0  &0&0                     \\\hline
\ZZ  &\ZZ&\ZZ                         \\\hline
\end{array}
\def\arraystretch{1.4}
\begin{array}{ l }
\\
\text{\tiny{$s=6$}}\\
\text{\tiny{$s=5$}}\\
\text{\tiny{$s=4$}}\\
\text{\tiny{$s=3$}}\\
\text{\tiny{$s=2$}}\\
\text{\tiny{$s=1$}}\\
\text{\tiny{$s=0$}}
\end{array}
\end{equation*}
The morphism 
$$H^2(\PP_sS;\ZZ)\cong IH_{\ov 0}^2(\PP_sS;\ZZ)\lra IH_{\ov m}^2(\PP_sS;\ZZ)\cong H^2(S;\ZZ)$$
determines up to sign a class $\pm w\in H^2(S;\ZZ)$ given by the image of a generator of
$H^2(\PP_sS;\ZZ)$. We get line bundles $L_s^+$ and $L_s^-$ over $S$
satisfying $c_1(L_S^\pm)=\pm w$. Since these two bundles are isomorphic as rank 2 vector bundles,
their Thom spaces $Th(L_S^\pm)\cong \PP_cS$ are homeomorphic.

Assume that we have an isomorphism $\Psi:IH_{\ov\bullet}(\PP_cS;\ZZ)\lra IH_{\ov\bullet}(\PP_cS';\ZZ)$.
Then the intersection forms of $S$ and $S'$ are equivalent, and it follows form Freedman's 
Theorem that $S$ and $S'$ are homeomorphic.
From the commutative diagram
$$
\xymatrix{
\ar[d]_\Psi IH_{\ov 0}^2(\PP_cS;\ZZ)\ar[r]&IH_{\ov m}^2(\PP_cS;\ZZ)\ar[d]^\Psi\\
IH_{\ov 0}^2(\PP_cS';\ZZ)\ar[r]&IH_{\ov m}^2(\PP_cS';\ZZ)
}
$$
we deduce that $\PP_cS$ and $\PP_cS'$ are homeomorphic.
\end{proof}

\begin{example}
Let $S$ be a surface of degree $4$ in $\CC\PP^3$, let $S'$ be the intersection of a quadric
 and a cubic in $\CC\PP^4$, and let $S''$ be the intersection of three quadrics in $\CC\PP^5$.
All three surfaces are examples of K3-surfaces with different intersection cohomology algebras.
Hence their projective cones are non-homeomorphic.
\end{example}

Let $S$ be a simply-connected 4-dimensional
smooth manifold and let $w\in H^2(S;\ZZ)$.
To such a pair $(S,\pm w)$  one can associate two homeomorphic Thom spaces
$Th(L_w^\pm)$.
The proof of Proposition $\ref{proj_cone_homeo}$ is easily generalized 
to this setting. We have:

\begin{prop}\label{equiv_pairs}
 Let $(S,\pm w)$ and $(S',\pm w')$ be two pairs. The following are equivalent:
 \begin{enumerate}[(1)]
  \item   The pairs are topologically equivalent: there is a homeomorphism $\phi:S\to S'$ 
such that $\phi^*(w)=w'$. 
  \item  The line bundles $L_w^\pm$ and $L_{w'}^\pm$ are isomorphic as real vector bundles.
  \item The Thom spaces $Th(L_w^\pm)$ and $Th(L_{w'}^\pm)$ are homeomorphic.
  \item  The integral intersection cohomologies $IH^*_{\ov\bullet}(Th(L_w^\pm);\ZZ)$ and
  $IH^*_{\ov\bullet}(Th(L_{w'}^\pm);\ZZ)$ are isomorphic as perverse graded algebras.
 \end{enumerate}
\end{prop}

\subsection{Perverse differential graded algebras}\label{perversecdgas}
As in the classical rational homotopy theory of Sullivan \cite{Su}, the study of rational intersection homotopy 
types is closely related to the homotopy theory of
perverse differential graded algebras. We next recall
the main definitions. Given our interest in varieties with only isolated singularities, we 
restrict to the particular case where perversities 
are given by a single integer, and refer \cite{Hov2} and \cite{CST} for the general definitions, 
in which perversities are given by multi-indexes.
For the rest of this section we let $\kk$ be a field of characteristic 0.

\begin{defi}\label{pcdga_def}
A \textit{perverse commutative differential graded algebra} \textit{over $\kk$}
 is a commutative monoid in the category of functors from $\widehat\Pp$
to the category $C^+(\mathbf{Vect_\kk})$ of cochain complexes of $\kk$-vector spaces:
this is a bigraded $\kk$-vector 
space $A_{\ov\bullet}^*=\{A^i_{\ov{p}}\}$, with $i\geq 0$ and $\ov{p}\in \widehat\Pp$,
together with a linear differential $d:A^i_{\ov{p}}\to A^{i+1}_{\ov{p}}$,
an associative product $\mu:A^i_{\ov{p}}\otimes A^j_{\ov{q}}\to A^{i+j}_{\ov{p}+\ov{q}}$ with unit $\eta:\kk\to A^0_{\ov{0}}$ 
and a poset map $A^i_{\ov{q}}\to A^i_{\ov{p}}$ for every $\ov{q}\leq \ov{p}$. 
Products and differentials satisfy the usual graded commutativity and graded Leibnitz rules, 
and are compatible with poset maps:
for all $\ov{p}\leq \ov{p}'$ and $\ov{q}\leq \ov{q}'$ the following diagrams commute:
$$\xymatrix{
A_{\ov p}\otimes A_{\ov q}\ar[d]\ar[r]^\mu& A_{\ov {p}+\ov{q}}\ar[d]\\
A_{\ov p'}\otimes A_{\ov q'}\ar[r]^\mu& A_{\ov{p}'+ \ov{q}'}
}\,\,\,\,\,\,;\,\,\,\,\,\,
\xymatrix{
A_{\ov p}\ar[d]\ar[r]^d& A_{\ov {p}}\ar[d]\\
A_{\ov p'}\ar[r]^d& A_{\ov {p'}}
}.
$$
\end{defi}
The cohomology of a perverse cdga naturally inherits the structure of a perverse commutative graded algebra.
Denote by $\pdga{\kk}$ the category of perverse cdga's over $\kk$.

\begin{defi}
A morphism of perverse cdga's $f:A_{\ov\bullet}\to B_{\ov\bullet}$ is called \textit{quasi-isomorphism}
if for every perversity $\ov p\in\widehat\Pp$
the induced map $H^*(A_{\ov p})\to H^*(B_{\ov q})$ is an isomorphism.
\end{defi}

The category $\pdga{\kk}$ admits a Quillen model structure with quasi-isomorphisms as weak equivalences and
surjections as fibrations (see \cite{Hov2}). The existence and uniqueness of minimal models of perverse cdga's
\`{a} la Sullivan
is proven in \cite{CST}.
Denote by $\Ho(\pdga{\kk})$ the homotopy category of perverse cdga's, defined by inverting quasi-isomorphisms.

\begin{defi}
 A perverse cdga $A_{\ov\bullet}$ is said to be \textit{intersection-formal} if there is an isomorphism in $\Ho(\pdga{\kk})$
from $A_{\ov\bullet}$ to $H^*(A_{\ov \bullet})$.
\end{defi}
Note that if a perverse cdga $A_{\ov{\bullet}}$ is intersection-formal, then both $A_{\ov{0}}$ and $A_{\ov{\infty}}$
are formal cdga's.

We shall consider the following weaker notion of intersection-formality, which excludes the infinite perversity.
Denote by $\GMpdga{\kk}$ the category of \textit{GM-perverse cdga's} defined by replacing $\widehat\Pp$ by $\Pp$ 
in Definition $\ref{pcdga_def}$.
Note that for a GM-perverse cdga $A_{\ov\bullet}$ the products
$A_{\ov p}\otimes A_{\ov q}\lra A_{\ov p + \ov q}$ need only be defined whenever $\ov p+\ov q<\ov \infty$.
The prefix ``GM'' accounts for Goresky-MacPherson, since only finite perversities are involved.
Denote by $\mathrm{U}:\pdga{\kk}\lra \GMpdga{\kk}$ the forgetful functor.

\begin{defi}
A perverse cdga $A_{\ov\bullet}$ is said to be \textit{GM-intersection-formal} if
there is an isomorphism in $\Ho(\GMpdga{\kk})$ from $A_{\ov\bullet}$ to $H^*(A_{\ov \bullet})$.
\end{defi}
Note that if a $A_{\ov{\bullet}}$ is GM-intersection-formal, then $A_{\ov{0}}$ is formal, but $A_{\ov{\infty}}$ 
need not be formal.
We remark that intersection-formality implies the vanishing of Massey products in intersection cohomology,
while GM-intersection-formality implies the vanishing of Massey products in $\mathrm{U}(IH^*_{\ov\bullet}(A))$.
We refer to $\S$3 of \cite{CST} for a proof of these statements and further discussion on (GM)-intersection-formality.

\subsection{Perverse algebraic model}
We next describe the perverse algebraic model of a complex projective variety with only isolated singularities,
as introduced in $\S$3.2 of \cite{CST}.

Let us first fix some notation.
Denote by $\Lambda(t,dt)=\kk(t,dt)$ the free cdga over $\kk$ generated by $t$ in degree 0 and $dt$ in degree 1.
For $\lambda\in\kk$ denote by $\delta_\lambda:\Lambda(t,dt)\to \kk$ the evaluation map defined 
by $t\mapsto \lambda$ and $dt\mapsto 0$.
Given a perversity $\ov p\in \widehat{\Pp}$, we will denote by
$\xi_{\leq \ov p}A(t,dt)$
the truncation 
of $A(t,dt)=A\otimes \Lambda(t,dt)$
by perverse degree $\ov p$, given in degree $k$ by:
$$\xi_{\leq \ov p}A(t,dt)^k=\left\{
\begin{array}{ll}
A^k\otimes\Lambda(t)\oplus A^{k-1}\otimes\Lambda(t)\otimes dt&,\text{ if } k<p\\
\Ker (d^k)\oplus A^k\otimes\Lambda(t)\otimes t\oplus A^{k-1}\otimes\Lambda(t)\otimes dt&,\text{ if } k=p\\
A^k\otimes\Lambda(t)\otimes t\oplus A^{k-1}\otimes\Lambda(t)\otimes dt&,\text{ if } k>p
\end{array}
\right..
$$
This truncation is compatible with differentials, products and poset maps: 
$$d(\xi_{\leq \ov p})\subseteq \xi_{\leq \ov {p}}\text{ and }\xi_{\leq \ov p}\times \xi_{\leq \ov q}\subseteq \xi_{\leq \ov {p}+\ov{q}}
\text{ for all }\ov p,\ov q\in\widehat\Pp, \text{ and }
\xi_{\leq \ov {q}}\subseteq \xi_{\leq \ov p}\text{ for all }\ov q\leq \ov p.$$

\begin{defi}\label{Pullbackpervers}
Let $f:A\to B$ be a morphism of cdga's over $\kk$. Given a perversity $\ov p\in\widehat\Pp$, consider the pull-back
in the category of complexes of $\kk$-vector spaces:
$$
\xymatrix{
\pb\ar[d]
\Ii_{\ov{p}}(f)\ar[r]&\xi_{\leq \ov p}B(t,dt)\ar[d]^{\delta_1}\\
A\ar[r]^{f}&B
}.
$$
Since $\xi_{\leq\ov p}$ is compatible with differentials, products and poset maps, $\Ii_{\ov{\bullet}}(f)$ with the products
and differentials
defined component-wise, is a perverse cdga, called the \textit{perverse cdga associated with $f$}.
\end{defi}

Let $X$ be a complex projective variety with only isolated singularities.
Let $T$ be a closed algebraic neighborhood of the singular locus $\Sigma$ in $X$
(in such a way that the inclusion $\Sigma\subset T$ is a homotopy equivalence, see \cite{Durfee2}).
Then the link of $\Sigma$ in $X$ is $L:=\partial T\simeq T^*:=T-\Sigma$.
The inclusion $\iota:L\hookrightarrow X_{reg}$ of the link into
the regular part of $X$ induces a morphism $\iota^*:\Aa_{pl}(X_{reg})\to \Aa_{pl}(L)$
of cdga's over $\QQ$, between the rational algebras of piecewise linear forms of $X_{reg}$ and $L$. 

\begin{defi}
The \textit{perverse algebraic model for $X$} is the rational perverse cdga
$I\Aa_{\ov{\bullet}}(X):=\Ii_{\ov{\bullet}}(\iota^*)$ associated with the morphism $\iota^*$. It is
given by the pull-back diagrams
$$
\xymatrix{
\pb\ar[d]
I\Aa_{\ov{p}}(X)\ar[r]&\xi_{\leq \ov p}\Aa_{pl}(L)(t,dt)\ar[d]^{\delta_1}\\
\Aa_{pl}(X_{reg})\ar[r]^{\iota^*}&\Aa_{pl}(L)
}.
$$
\end{defi}
We have an isomorphism of perverse commutative graded algebras $H^*(I\Aa_{\ov \bullet}(X))\cong IH^*_{\ov\bullet}(X;\QQ)$.
For the $\ov{0}$-perversity
we have a quasi-isomorphism of cdga's $I\Aa_{\ov0}(X)\simeq \Aa_{pl}(\overline{X})$, where $\overline{X}\to X$ is a normalization of $X$.
For the $\ov\infty$-perversity we recover the rational homotopy type
$I\Aa_{\ov{\infty}}(X)\simeq \Aa_{pl}({X}_{reg})$ of the regular part of $X$.

The above construction defines a contravariant functor $I\Aa_{\ov\bullet}:\Vv_\CC\lra \Ho(\pdga{\QQ})$
from the category $\Vv_\CC$ of complex projective varieties with only isolated singularities 
and stratified morphisms, to the the homotopy category of perverse cdga's over $\QQ$.

\begin{defi}Let $\QQ\subset \mathbf{K}$ be a field.
A complex projective variety $X$ with isolated singularities is called \textit{(GM)-intersection-formal
over $\mathbf{K}$} if and only if $I\Aa_{\ov\bullet}(X)\otimes \mathbf{K}$ is (GM)-intersection-formal.
\end{defi}

\section{Mixed Hodge Structures and Perverse Weight Spectral Sequence}\label{Section_MHS}
In this section, we endow the perverse algebraic model of a complex projective variety $X$ with only isolated singularities,
with natural mixed Hodge structures.
We then study the perverse weight spectral sequence of $X$
and prove that
the complex intersection homotopy type of $X$ is a direct consequence of 
its perverse weight spectral sequence.
Lastly, we describe the perverse weight spectral sequence in terms of the cohomologies
of the varieties associated with a resolution of $X$.

\subsection{Mixed Hodge structures on intersection cohomology}
Deligne showed that the rational cohomology ring
of every complex algebraic variety $X$ is endowed with \textit{mixed Hodge structures}:
for every $k\geq 0$, there is an increasing filtration
$W$ of the rational cohomology $H^k(X;\QQ)$,
called the \textit{weight filtration}, together with a decreasing filtration 
$F$ of the complex cohomology $H^k(X;\CC)$, called the \textit{Hodge filtration}, 
in such a way that the filtration induced by $F$ and its complex conjugate $\overline{F}$ on
the graded objects
$Gr_m^WH^k(X;\CC)\cong Gr_m^WH^k(X;\QQ)\otimes\CC$
define a Hodge decomposition of pure weight $m$.
Furthermore, these filtrations are functorial and
compatible with products of varieties (we refer to \cite{DeHII}, \cite{DeHIII} or the book \cite{PS} for details).

If $X$ is a complex projective variety with only isolated singularities,
the compatible mixed Hodge structures on the cohomologies
of $X$ and $X_{reg}$ define canonical mixed Hodge structures on
 $IH^k_{\ov{p}}(X;\QQ)$, which are compatible with products and
poset maps. In particular, for every $k\geq 0$ the morphism
 $IH^k_{\ov{0}}(X;\QQ)\to IH^k_{\ov{\infty}}(X;\QQ)$ induced by the inclusion
$X_{reg}\hookrightarrow \ov{X}$ preserves mixed Hodge structures.

A well-known result on the mixed Hodge theory of projective varieties with isolated singularities
is that for the middle perversity, 
the weight filtration $W$ on $IH^k_{\ov{m}}(X;\QQ)$ is pure of weight $k$, for all $k\geq 0$, that is:
$0=W_{k-1}\subset W_k=IH^k_{\ov{m}}(X;\QQ).$
This is a consequence of Gabber's purity theorem and the decomposition theorem of
intersection homology (see \cite{Ste}. A direct proof using Hodge theory appears in \cite{Na2}).
We next give the bounds on the weight filtration $W$ for an arbitrary perversity.
\begin{lem}
Let $X$ be a complex projective variety of dimension $n$ with only isolated singularities.
\begin{enumerate}[(1)]
\item  If $\ov p<n-1$ then $0=W_{k-1}\subset W_k\subset \cdots \subset W_{2k}=IH^k_{{\ov p}}(X;\QQ)$. If in addition,
 $k\leq \ov p+1$ or $k>n$, then the weight filtration $W$ on $IH^k_{{\ov p}}(X;\QQ)$ is pure of weight $k$.
 \item  If $\ov p=n-1$ then the weight filtration $W$ on $IH^k_{{\ov p}}(X;\QQ)$ is pure of weight $k$ for, all $k\geq 0$.
\item If $\ov p>n-1$ then $0=W_{-1}\subset W_0\subset \cdots \subset W_{k}=IH^k_{{\ov p}}(X;\QQ)$.
If in addition, $k<n$ or $k\geq \ov p+1$, then the weight filtration $W$ on $IH^k_{{\ov p}}(X;\QQ)$ is pure of weight $k$.
\end{enumerate}
\end{lem}
\begin{proof}
The weight filtration on the cohomologies of $X$ and $X_{reg}$ is bounded respectively by:
$$0=W_{-1}\subset W_0\subset\cdots\subset W_k=H^k(X;\QQ)\text{ and }0=W_{k-1}\subset W_k\subset\cdots\subset W_{2k}=H^k(X_{reg};\QQ).$$
By Theorem 2.3.5 of \cite{DeHII} the category of mixed Hodge structures is abelian and morphisms of mixed Hodge
structures are strictly compatible with filtrations.
Hence $\Img(H^k(X;\QQ)\lra H^k(X_{reg};\QQ))$ carries a pure Hodge structure of weight $k$, for all $k\geq 0$.
Since $X$ has only isolated singularities, for $k>n$, the filtration $W$ on $H^k(X;\QQ)$ is pure of weight $k$,
while for $k<n$, the filtration $W$ on $H^k(X_{reg};\QQ)$ is pure of weight $k$ (see Theorem 1.13 of \cite{Ste}).
The result is now a matter of verification.
\end{proof}

\subsection{Mixed Hodge perverse cdga's}
We next introduce mixed Hodge perverse cdga's. These are perverse cdga's carrying
compatible mixed Hodge structures in each perversity and degree.
Using Deligne's splitting of mixed Hodge structures we show that, over the
complex numbers, certain mixed Hodge perverse cdga's are isomorphic
to the perverse cdga defined by the first term
of the spectral sequence associated with the weight filtration.
\begin{defi}
A \textit{filtered perverse cdga} $(A_{\ov\bullet},W)$ is a perverse cdga $A_{\ov\bullet}$ together with a filtration
$\{W_{m}A_{\ov\bullet}\}$ indexed by the integers and satisfying: 
\begin{enumerate}[(i)]
 \item  $W_{m-1}A_{\ov p}\subset W_mA_{\ov p}$ and
$d(W_{m}A_{\ov p})\subset W_mA_{\ov p}$ for all $m\in\ZZ$ and all $\ov{p}\in \widehat\Pp$,
\item $W_mA_{\ov p}\cdot W_nA_{\ov q}\subset W_{m+n}A_{\ov p+\ov q}$ for all $m,n\in\ZZ$ and all $\ov{p},\ov{q}\in\widehat\Pp$,
\item $W_mA_{\ov q}\subset W_mA_{\ov p}$, for all $m\in\ZZ$ and all $\ov q\leq \ov p$,
\item For all $n\geq 0$ and all $\ov{p}\in\widehat \Pp$ there exists integers $m$ and $l$ such that
$W_mA^n_{\ov p}=0$ and $W_lA^n_{\ov p}=A^n_{\ov p}$.
\end{enumerate}
\end{defi}

The spectral sequence associated with a filtered perverse cdga $(A_{\ov\bullet},W)$ is compatible with the multiplicative
structure. Hence for all $r\geq 0$, the term $E_r(A_{\ov\bullet},W)$ is a perverse differential 
bigraded algebra with differential $d_r$ of bidegree $(r,1-r)$.

\begin{defi}
A \textit{mixed Hodge perverse cdga} is a filtered perverse cdga $(A_{\ov\bullet},W)$ over $\QQ$,
together with a filtration $F$ on $A_{\ov\bullet}\otimes\CC$, such that for each $n\geq 0$ and each $\ov p\in\widehat \Pp$ the triple
$(A^n_{\ov p},\Dec W,F)$
is a mixed Hodge structure and the differentials $d:A^n_{\ov p}\to A^{n+1}_{\ov p}$, products
$A^n_{\ov p}\times A^m_{\ov q}\to A^{n+m}_{\ov p+\ov q}$ and poset maps $A^n_{\ov {q}}\to A^n_{\ov p}$
are compatible with $W$ and $F$.
\end{defi}
By an abuse of notation, we shall denote such a mixed Hodge perverse cdga by a triple $(A_{\ov\bullet},W,F)$, noting that the second filtration $F$
is not defined over $A_{\ov\bullet}$, but on its complexification $A_{\ov\bullet}\otimes\CC$.
The filtration $\Dec W$ denotes Deligne's d\'{e}calage of the weight filtration $W$ (see Definition 1.3.3 of \cite{DeHII}),
given by $\Dec W_p A_{\ov\bullet}^n:= W_{p-n}A^{n}_{\ov\bullet} \cap d^{-1}(W_{p-n-1}A^{n+1}_{\ov\bullet})$.
By forgetting the perversities we recover the notion of mixed Hodge cdga appearing in \cite{CG1}.
This differs form Morgan's original definition (see \cite{Mo}) by a shift, which we introduce to make it compatible with 
Deligne's mixed Hodge complexes.

\begin{lem}\label{splittingQ}
Let $(A_{\ov{\bullet}},W,F)$ be a mixed Hodge perverse cdga such that $d(W_pA_{\ov{\bullet}})\subset W_{p-1}A_{\ov{\bullet}}$.
There is an isomorphism of complex perverse cdga's
$A_{\ov{\bullet}}\otimes\CC\cong E_1(A_{\ov\bullet}\otimes\CC,W)$.
\end{lem}
\begin{proof}
The proof is an adaptation to the perverse setting of Lemma 3.20 of \cite{CG1} for mixed Hodge cdga's, see also \cite{Mo}. We indicate the main steps.
Since the triple $(A_{\ov{p}}^n,\Dec W,F)$ is a mixed Hodge structure, by Lemma
1.2.11 of \cite{DeHII}, there are functorial decompositions
$$A^n_{\ov{p}}\otimes\CC=\bigoplus V^{i,j}_{\ov{p},n},\text{ with }\Dec W_m(A^n_{\ov p}\otimes\CC)=\bigoplus_{i+j\leq m} V^{i,j}_{\ov{p},n}.$$
Since the differentials, products and poset maps of $A_{\ov\bullet}$ are morphisms of mixed Hodge structures,
these decompositions are compatible with the perverse cdga structure.
Define complex vector spaces $$A^{i,n-i}_{\ov p}:=\bigoplus_{r} V_{\ov p,n}^{n-r-i,r}.$$
This gives a decomposition $A_{\ov \bullet}^n\otimes\CC=\bigoplus A^{i,n-i}_{\ov p}$ compatible with products and poset maps, and such that
$d A^{i,n-i}_{\ov{p}}\subset A^{i+1,n-i}_{\ov{p}}$. Since
$d(W_pA_{\ov{\bullet}})\subset W_{p-1}A_{\ov{\bullet}}$, it follows that
$W_pA_{\ov{\bullet}}^n=\Dec W_{p+n}A_{\ov{\bullet}}^n$. Then
$$W_m(A_{\ov p}^n\otimes \CC)=\bigoplus_{i\leq m} A^{-i,n+i}_{\ov p}.$$
Since $Gr^W_m(A_{\ov p}^n\otimes \CC)=A^{-m,n+m}_{\ov p}$ and $d A^{m,n-m}_{\ov{p}}\subset A^{m+1,n-m}_{\ov{p}}$,
the differential of $E_r(A_{\ov\bullet},W)$ is 
is trivial for all $r\neq 1$. This gives an isomorphism
$\pi:A_{\ov \bullet}\otimes\CC\to E_1(A_{\ov \bullet}\otimes\CC,W)$
of perverse cdga's such that
$\pi(A^{-m,n+m}_{\ov p})\cong Gr_m^WA_{\ov p}^n=E_1^{-m,n+m}(A_{\ov p}\otimes\CC,W).$
\end{proof}

\begin{rmk}
Let $(A,W)$ be a filtered cdga over a field $\kk$ and let $\kk\subset\mathbf{K}$ be a field extension. If $A$ has finite type
then by Theorem 2.26 of \cite{CG1} we have that
$A\cong E_r(A,W)$ if and only 
if $A\otimes_\kk\mathbf{K}\cong E_r(A\otimes_\kk\mathbf{K},W)$.
The same proof is valid for perverse cdga's of finite type. 
Hence in this case, the isomorphism of Lemma $\ref{splittingQ}$ descends to an
isomorphism over $\QQ$.
\end{rmk}

Consider on $\QQ(t,dt)$ the b\^{e}te filtration $\sigma$.
This is the multiplicative filtration defined by setting
$t$ of weight $0$ and $dt$ of weight $-1$.
Endow $\CC(t,dt)=\QQ(t,dt)\otimes\CC$ with the b\^{e}te filtration $\sigma$ and the trivial filtration $t$.
Since $\Dec \sigma=t$, the triple $(\QQ(t,dt), \sigma, t)$ is a mixed Hodge cdga,
which induces on $H^0(\QQ(t,dt))\cong \QQ$ the trivial mixed Hodge structure.

\begin{prop}\label{pbMHdga}
Let $f:(A,W,F)\lra (B,W,F)$ be a morphism of mixed Hodge cdga's.
Then the perverse cdga $\Ii_{\ov\bullet}(f)$ of Definition $\ref{Pullbackpervers}$
with the filtrations $W$ and $F$ defined via the pull-backs
 $$
\xymatrix{
\pb\ar[d]
(\Ii_{\ov{p}}(f),W,F)\ar[r]&(\xi_{\leq \ov p}B(t,dt),W*\sigma,F*t)\ar[d]^{\delta_1}\\
(A,W,F)\ar[r]^{f}&(B,W,F)
}.
$$
is a mixed Hodge perverse cdga, where the filtrations $W*\sigma$ and $F*t$ are defined by convolution:
$$(W*\sigma)_pB(t,dt)=
W_pB\otimes\Lambda(t)+W_{p+1}B\otimes\Lambda(t)\otimes dt
\text{ and }(F*t)^pB(t,dt)=F^pB\otimes\Lambda(t,dt).$$

\end{prop}
\begin{proof}
It suffices to verify that for all $n\geq 0$ and all $\ov p\in\widehat \Pp$, the triple
$(\Ii_{\ov{p}}(f)^n,\Dec W,F)$ is a mixed Hodge structure, and that the products
and poset maps of $\Ii_{\ov{\bullet}}(f)$ are compatible with the filtrations $W$ and $F$.
Since d\'{e}calage commutes with pull-backs and $\Dec(W*\sigma)=\Dec W* t$, we have
 $$
\xymatrix{
\pb\ar[d]
(\Ii_{\ov{p}}(f)^n,\Dec W,F)\ar[r]&(\xi_{\leq \ov p}B(t,dt)^n,\Dec W*t,F*t)\ar[d]^{\delta_1}\\
(A^n,\Dec W,F)\ar[r]^{f}&(B^n,\Dec W,F)
}.
$$
Since the category of mixed Hodge structures is abelian, the truncations
$(\xi_{\leq \ov p}B(t,dt)^n,\Dec W,F)$
are mixed Hodge structures, and hence the above pull-back gives a mixed Hodge structure.
It is straightforward to verify that the products
$\xi_{\leq \ov q}B(t,dt)\times \xi_{\leq \ov p}B(t,dt)\to \xi_{\leq \ov q+\ov p}B(t,dt)$
and poset maps
$\xi_{\leq \ov q}B(t,dt)\lra \xi_{\leq \ov p}B(t,dt)$ for $\ov q\leq \ov p$,
are compatible with filtrations, so that the perverse algebra structure of 
$\Ii_{\ov{\bullet}}(f)$ is also compatible with $W$ and $F$.
\end{proof}

\begin{lem}\label{commutepb}
Let $f:(A,W,F)\to (B,W,F)$ be a morphism of mixed Hodge cdga's. There is an isomorphism of perverse differential bigraded algebras 
$E_1(\Ii_{\ov{\bullet}}(f),W)\cong \Ii_{\ov{\bullet}}(E_1(f,W)).$
\end{lem}
\begin{proof}
The evaluation map
$\delta_1:(\xi_{\leq \ov p}B(t,dt),W*\sigma)\lra (B,W)$ induces a surjective quasi-isomorphism at the level of $E_1$,
for any perversity $\ov p$.
Therefore we have
$E_1(\Ker(f-\delta_1))=\Ker(E_1(f-\delta_1)$
(c.f. Proposition
3.9 of \cite{Cirici}). 
It remains to observe that $E_1$ commutes with the truncations $\xi_{\leq \ov p}$ and that we have
a canonical isomorphism of differential bigraded algebras
$E_1(B(t,dt), W*\sigma)\cong E_1(B,W)(t,dt).$
\end{proof}

\subsection{Mixed Hodge structures on the perverse algebraic model}
We next show that the perverse algebraic model of a complex projective variety $X$ with only isolated singularities
carries well-defined mixed Hodge structures (in the homotopy category) which are compatible
with the mixed Hodge structures on the rational homotopy types of
$\ov{X}$ and $X_{reg}$, and are functorial for stratified morphisms.

We first recall some basic definitions for the theory of mixed Hodge structures 
in rational homotopy. We refer to \cite{Mo, Na, CG1} for further details.
\begin{defi}
A \textit{mixed Hodge diagram (of cdga's over $\QQ$)} consists of
a filtered cdga $(A_{\QQ},W)$ over $\QQ$, a bifiltered cdga $(A_{\CC},W,F)$ over $\CC$, together with
a string of filtered quasi-isomorphisms
from $(A_\QQ,W)\otimes \CC$ to $(A_\CC,W).$
In addition, the following axioms are satisfied:
\begin{enumerate}
\item[($\mathrm{MH}_0$)] The weight filtrations $W$ are regular and exhaustive. The Hodge filtration $F$ is biregular.
The cohomology $H(A_\QQ)$ has finite type.
\item[($\mathrm{MH}_1$)] For all $p\in\ZZ$, the differential of $Gr_p^WA_\CC$ is strictly compatible with $F$.
\item[($\mathrm{MH}_2$)] For all $n\geq 0$ and all $p\in\ZZ$, the filtration $F$ induced on $H^n(Gr^W_pA_{\CC})$ defines a pure Hodge structure of
weight $p+n$ on $H^n(Gr^W_pA_\QQ)$.
\end{enumerate}
\end{defi}

Morphisms of mixed Hodge diagrams are given by level-wise morphisms of (bi)filtered cdga's making the corresponding diagrams strictly commutative.
Note that by forgetting the multiplicative structures we recover the original notion of mixed Hodge complex introduced by Deligne
(see 8.1 of \cite{DeHIII}).

Axiom ($\mathrm{MH}_2$) implies that for all $n\geq 0$ the triple $(H^n(A_\QQ),\Dec W,F)$ is a mixed Hodge structure over $\QQ$.
In particular, the cohomology of every mixed Hodge diagram is a mixed Hodge cdga with trivial
differential. Since the category of mixed Hodge structures is abelian, every mixed Hodge cdga is a mixed Hodge diagram in which the
comparison quasi-isomorphisms are identities.

\begin{defi}
Let $X$ be a topological space. A \textit{mixed Hodge diagram for $X$} is a mixed Hodge diagram 
$\Aa(X)$ whose rational component $\Aa(X)_\QQ\simeq A_{pl}(X)$ is quasi-isomorphic to the rational algebra of piecewise linear forms of $X$.
\end{defi}
Note that when such a mixed Hodge diagram $\Aa(X)$ exists, the above quasi-isomorphism induces an isomorphism
$H^*(\Aa(X)_\QQ)\cong H^*(X;\QQ)$ endowing the cohomology of $X$ with mixed Hodge structures.

We now prove the main result of this section, endowing the perverse algebraic model of a projective variety with only isolated singularities,
with mixed Hodge structures. The proof relies on the existence of compatible mixed Hodge structures on the rational
homotopy types of the link of the singularities and the regular part of the variety respectively (\cite{Mo, Na, DH}),
together with the existence of minimal models \`{a} la Sullivan of morphisms of cohomologically 
connected mixed Hodge diagrams, which is proven in \cite{CG1} (see also \cite{Cirici} for a homotopical framework of such models).

\begin{teo}\label{MHSmodel}
Let $X$ be a complex projective variety with only isolated singularities and let $I\Aa_{\ov\bullet}(X)$ denote the perverse algebraic model of $X$.
There is a mixed Hodge perverse cdga $(IM_{\ov\bullet}(X),W,F)$ and a string
$IM_{\ov\bullet}(X)\leftarrow \ast \to I\Aa_{\ov\bullet}(X)$ of quasi-isomorphisms of perverse cdga's such that:
\begin{enumerate}[(1)]
\item $IM_{\ov\bullet}(X)=\Ii_{\ov\bullet}(\wt \iota)$, where $\wt \iota:M(X_{reg})\to M(L)$ is a 
model of mixed Hodge cdga's for the rational homotopy type of the inclusion $\iota:L\hookrightarrow X_{reg}$.
\item There is an isomorphism of mixed Hodge perverse graded algebras $H^*(IM_{\ov\bullet}(X))\cong IH^*_{\ov\bullet}(X;\QQ)$.
\item The mixed Hodge cdga's $IM_{\ov{0}}(X)$ and  $IM_{\ov{\infty}}(X)$ 
define the mixed Hodge structures on the rational homotopy type of a normalization $\ov{X}$ of $X$ and the regular part $X_{reg}$ of $X$
respectively.
\item The differential of $IM_{\ov\bullet}(X)$ satisfies $d(W_pIM_{\ov\bullet}(X))\subset W_{p-1}IM_{\ov\bullet}(X)$.
\end{enumerate} 
This construction defines a functor $IM_{\ov \bullet}:\Vv_\CC\lra \Ho(\widehat{\Pp}\mathbf{MH}\mathsf{CDGA})$
with values in the homotopy category of mixed Hodge perverse cdga's.
\end{teo}
\begin{proof}
By Theorem 3.2.1 of \cite{DH} (see also $\S$13 of \cite{Na}) there are mixed Hodge diagrams $\Aa(X_{reg})$ and $\Aa(L)$ 
for $X_{reg}$ and $L$ respectively, together with a morphism
$\Aa(X_{reg})\lra \Aa(L)$ 
whose rational component is the morphism
$\iota^*:\Aa_{pl}(X_{reg})\lra \Aa_{pl}(L)$ of rational piecewise linear forms induced by the inclusion $L\hookrightarrow X_{reg}$.
By Theorem 3.19 of \cite{CG1} on the existence of relative minimal models for mixed Hodge diagrams,
we can construct a commutative diagram of mixed Hodge diagrams
$$
\xymatrix{
\Aa(X_{reg})\ar[r]^{\iota^*}&\Aa(L)\\
\ar[u]\ar[d]\ast\ar[r]&\ast\ar[u]\ar[d]\\
M(X_{reg})\ar[r]^{\wt \iota}&M(L)
}
$$
where the vertical maps are quasi-isomorphisms and $\wt \iota$ is a morphism of mixed Hodge cdga's
whose differential satisfies $d(W_p)\subset W_{p-1}$.
By Proposition $\ref{pbMHdga}$, the perverse cdga $IM_{\ov{\bullet}}(X):=\Ii_{\ov\bullet}(\wt i)$
associated with $\wt \iota$ is a mixed Hodge cdga.
Furthermore, the above commutative
diagram gives a string of quasi-isomorphisms of perverse cdga's from 
$IM_{\ov\bullet}(X)$ to $I\Aa_{\ov\bullet}(X)$.
This proves (1).
The induced map $\wt \iota:M(X_{reg})\lra M(L)$ induces 
a morphism of mixed Hodge structures in cohomology $H^*(X_{reg})\lra H^*(L)$.
Hence (2) follows from the isomorphism $\Ker(H^*(X_{reg})\lra H^*(L))\cong \Img(H^*(X)\lra H^*(X_{reg})$.
Assertion (3) is easily verified. Lastly, (4) follows from the fact that
the differential on 
$IM_{\ov p}(X)$ is defined component-wise by differentials satisfying $d(W_p)\subset W_{p-1}$.
The construction of $IM_{\ov \bullet}(X)$ is functorial (in the homotopy category) for morphisms $f:X\to Y$ such that $f(X_{reg})\subset Y_{reg}$.
\end{proof}

\subsection{Perverse weight spectral sequence}
Let $X$ be a complex projective variety with only isolated singularities.
The inclusion $\iota:L\hookrightarrow X_{reg}=X-\Sigma$ of the link into the regular part of $X$
induces a morphism of (multiplicative) weight spectral sequences $E_1(\iota^*):E_1(X_{reg})\lra E_1(L)$,
where $E_1(X_{reg})$ and $E_1(L)$
are the spectral sequences associated with the weight filtration of a mixed Hodge diagram
for $X_{reg}$ and $L$ respectively.

\begin{defi}
The \textit{perverse weight spectral sequence of $X$} is the perverse differential bigraded algebra
$IE_{1,\ov\bullet}(X):=\Ii_{\ov\bullet}(E_1(\iota^*))$
associated with the morphism $E_1(\iota^*)$,
given by the pull-back diagrams
$$
\xymatrix{
\pb\ar[d]
IE_{1,\ov p}(X)\ar[r]&\xi_{\leq \ov p}E_1(L)(t,dt)\ar[d]^{\delta_1}\\
E_1(X_{reg})\ar[r]^{E_1(\iota^*)}&E_1(L)
}.
$$

\end{defi}
We remark that $IE_{1,\ov\bullet}(X)$ is only well-defined in the homotopy category of perverse differential bigraded algebras.
Its cohomology is a well-defined algebraic invariant of $X$ and satisfies
$$IE^{r,s}_{2,\ov p}(X):=H^{r,s}(IE_{1,\ov p}(X))\cong Gr_{s}^WIH_{\ov p}^{r+s}(X;\QQ).$$

The following result states that the complex intersection homotopy type of a projective variety with isolated singularities is
determined by its perverse weight spectral sequence.

\begin{teo}\label{IE1formality}Let $X$ be a complex projective variety with only isolated singularities. There is an isomorphism
from $I\Aa_{\ov\bullet}(X)\otimes\CC$ to $IE_{1,\ov\bullet}(X)\otimes\CC$
in the homotopy category $\Ho(\pdga{\CC})$.
\end{teo}
\begin{proof}
Let $(IM_{\ov\bullet}(X),W,F)$ be the mixed Hodge perverse cdga given by Theorem 
$\ref{MHSmodel}$.
Since the differential satisfies $d(W_pIM_{\ov\bullet}(X))\subset W_{p-1}IM_{\ov\bullet}(X)$,
by Lemma $\ref{splittingQ}$, we have an isomorphism of complex perverse cdga's
$IM_{\ov{\bullet}}(X)\otimes\CC\cong E_1(IM_{\ov\bullet}(X)\otimes\CC,W).$

By construction, we have that $IM_{\ov\bullet}(X)=\Ii_{\ov\bullet}(\wt \iota)$,
where $\wt \iota:(M(X_{reg}),W,F)\to (M(L),W,F)$ is a morphism of mixed Hodge cdga's computing the rational homotopy type of
$\iota:L\hookrightarrow X_{reg}$. Therefore by Lemma $\ref{commutepb}$, we have an isomorphism of perverse cdga's
$E_1(IM_{\ov\bullet}(X),W)\cong \Ii_{\ov{\bullet}}(E_1(\wt \iota,W)).$
It only remains to note that there is a string of quasi-isomorphisms from
$\Ii_{\ov{\bullet}}(E_1(\wt \iota))$ to $IE_{1\ov\bullet}(X)=\Ii_{\ov{\bullet}}(E_1(\iota^*))$.
\end{proof}

As a direct application we have the following "purity implies formality" result in the context of rational intersection homotopy:
\begin{cor}
Let $X$ be a complex projective variety with only isolated singularities.
If the weight filtration on $IH^{k}_{\ov p}(X)$ is pure of weight $k$,
for all $k\geq 0$ and every finite perversity $\ov p\in \Pp$,
then $X$ is GM-intersection-formal over $\CC$.
\end{cor}

\begin{example}
[$\QQ$-homology manifolds]
Let $X$ be a complex projective variety which is also a
$\QQ$-homology manifold. Then 
$IH^k_{\ov p}(X;\QQ)\cong IH^k_{\ov 0}(X;\QQ)$ for every finite perversity $\ov p\in\Pp$ (see 6.4 of \cite{GMP1}).
In particular, the weight filtration on $IH^k_{\ov p}(X;\QQ)$ is pure of weight $k$, for all $k\geq 0$
and every $\ov p\in \Pp$. Hence $X$ is GM-intersection-formal over $\CC$.
Examples of such varieties are given by weighted projective spaces or more generally $V$-manifolds
(see Appendix B of \cite{Di}),
surfaces with $A_1$-singularities, the Cayley cubic or the Kummer surface.
\end{example}

\subsection{Resolution of singularities and weight spectral sequence}
Let $X$ be a complex projective variety of dimension $n$ with only isolated singularities.
Denote by $\Sigma$ the singular locus of $X$ and by
$L:=L(\Sigma,X)$ the link of $\Sigma$ in $X$. 

Let $f:\wt X\lra X$ be a resolution of singularities of $X$ such that $D:=f^{-1}(\Sigma)$ is a simple normal crossings divisor.
We may write
$D=D_1\cup\cdots \cup D_N$ as the union
of irreducible smooth varieties meeting transversally. Let $D^{(0)}=\wt X$ and for all $r>0$, 
denote by $D^{(r)}=\bigsqcup_{|I|=r}D_I$ the disjoint union of all 
$r$-fold intersections 
$D_I:=D_{i_1}\cap\cdots \cap D_{i_r}$ where $I=\{i_1,\cdots,i_r\}$ denotes an ordered subset of $\{1,\cdots,N\}$.
Since $D$ has simple normal crossings, it follows that $D^{(r)}$ is a smooth projective variety of dimension $n-r$. 
For $1\leq k\leq r$, denote by $j_{I,k}:D_I\hookrightarrow D_{I\setminus \{i_k\}}$ 
the inclusion and let $j_{r,k}:=\bigoplus_{|I|=r} j_{I,k}:D^{(r)}\hookrightarrow D^{(r-1)}$.
These maps define a simplicial resolution
$D_\bullet=\{D^{(r)}, j_{r,k}\}$.

Let $r\geq 1$. For every $1\leq k\leq r$ we
will denote by $j_{r,k}^*:=(j_{r,k})^*:H^*(D^{(r-1)})\to H^*(D^{(r)})$ the restriction morphism induced by the inclusion $j_{r,k}$ and
by $\gamma_{r,k}:=(j_{r,k})_!:H^{*-2}(D^{(r)})\to H^*(D^{(r-1)})$ the corresponding Gysin map.
We have combinatorial restriction morphisms
$$j_{(r)}^s:=\sum_{k=1}^r (-1)^{k-1} (j_{r,k})^{*}:H^{s}(D^{(r-1)};\QQ)\lra H^s(D^{(r)};\QQ)$$
and combinatorial Gysin maps
$$\gamma_{(r)}^s:=\sum_{k=1}^r (-1)^{k-1} (j_{r,k})_{!}:H^{s-2r}(D^{(r)};\QQ)\lra H^{s-2(r-1)}(D^{(r-1)};\QQ).$$

With this notation, the weight spectral sequence for $X_{reg}$ can be written as:
$$
E_1^{r,s}(X_{reg})=
\def\arraystretch{1.6}
\begin{array}{c c c c c c c c c}
\multicolumn{1}{c}{}\\
\cdots&\lra &H^{s-4}(D^{(2)};\QQ)&\xra{\gamma^s_{(2)}}&H^{s-2}(D^{(1)};\QQ)&\xra{\gamma^s_{(1)}}&H^s(\wt X;\QQ)&\lra&0\\ 
\multicolumn{1}{c}{}&\multicolumn{1}{c}{}&\multicolumn{1}{c}{\text{\tiny{$r=-2$}}}&&\multicolumn{1}{c}{\text{\tiny{$r=-1$}}}&&\multicolumn{1}{c}{\text{\tiny{$r=0$}}}
\end{array}.
$$
Its algebra structure 
is given by the maps
$H^{m}(D^{(p)};\QQ)\otimes H^{l}(D^{(q)};\QQ)\lra H^{m+l}(D^{(p+q)};\QQ)$
induced by combinatorial restriction morphisms, for $p+q\leq n$ (see \cite{Mo}).

We next describe the multiplicative weight spectral sequence of the link $L\simeq L(D,\wt X)$.
In \cite{Durfee}, Durfee endows the cohomology of the link of an
isolated singularity with mixed Hodge structures, and describes
its weight spectral sequence in terms of a resolution of singularities.
However, such spectral sequence is not multiplicative,
since it is the spectral sequence associated with a mixed Hodge complex for $L$.
To describe the multiplicative weight spectral sequence of the link we
analyze the construction due to Durfee-Hain \cite{DH} 
of a mixed Hodge diagram of cdga's for $L$.

For all $i\in \{1,\cdots,N\}$ define 
$$L_i:=L(D_i-\bigsqcup_{i\neq j} D_i\cap D_j,\wt X).$$
For all $r>0$ denote by $L^{(r)}=\bigsqcup_{|I|=r}L_I$ the disjoint union of all 
$r$-fold intersections 
$L_I:=L_{i_1}\cap\cdots \cap L_{i_r}$ where $I=\{i_1,\cdots,i_r\}$ denotes an ordered subset of $\{1,\cdots,N\}$.
We have
$$L^{(1)}:=\bigsqcup_i L_i,\,\, L^{(2)}:=\bigsqcup_{i\neq j}L_i\cap L_j,\cdots$$
We obtain a
simplicial manifold
$L_\bullet=\{L^{(r)}, i_{r,k}\}$,
where $i_{r,k}:L^{(r)}\hookrightarrow L^{(r-1)}$, for
$1\leq k\leq r$, denote the natural inclusions.

The multiplicative weight spectral sequence for $L^{(r)}$ is given by:
$$E_1^{*,*}(L^{(r)})=\bigoplus_{I=\{i_1,\cdots,i_r\}} E_1(D_I-\mathrm{Sing}(D_I))\wt\otimes \Lambda(\theta_{i_1},\cdots,\theta_{i_r}).$$
where $\theta_k$ are generators of bidegree $(-1,2)$ 
and $\wt \otimes$ accounts for the fact that the differential of $\theta_k$ is given by
$d(\theta_{k})=c_k$, where $c_k\in H^2(D_{k};\QQ)$ is the Chern class of $D_{k}$.

The multiplicative weight spectral sequence for $L$ is then given by the end
$$E_1^{p,q}(L):=\int_\alpha \bigoplus_mE_1^{p-m,q}(L^{(\alpha)})\otimes \Omega_\alpha^m,$$
where $\Omega_\alpha$
is the simplicial cdga given by
$\Omega_\alpha:={{\Lambda(t_0,\cdots,t_\alpha,dt_0,\cdots,dt_\alpha)}/{\sum t_i-1,\sum dt_i}}$,
with $t_i$ of degree 0 and $dt_i$ of degree 1.

In Sections $\ref{perverseordinary}$ and $\ref{perversesurface}$ we provide a description
of the morphism $E_1(X_{reg})\lra E_1(L)$ in the particular cases of
ordinary isolated singularities and isolated surface singularities respectively,
thus giving an explicit description of the perverse weight spectral sequence in these cases.

\section{Ordinary Isolated Singularities}\label{Section_OIS}
For the rest of this section, let $X$ be a complex projective variety of dimension $n$ with isolated singularities. 
We will show that if $X$ admits a resolution of singularities 
in such a way that the exceptional divisor is smooth, and if the link
of each singular point is  $(n-2)$-connected, then $X$ is GM-intersection-formal over $\CC$.
The main class of examples to which this result applies are varieties with 
ordinary multiple points, but it also applies to a large family of
hypersurfaces with isolated singularities and more generally, to complete
intersections with isolated singularities
admitting  a resolution of singularities with smooth exceptional divisor.

\subsection{Notation}
Denote by $\Sigma$ the singular locus of $X$ and by $X_{reg}=X-\Sigma$ its regular part.
Denote by $L:=L(\Sigma,X)$ the link of $\Sigma$ in $X$, and by $\iota:L\hookrightarrow X_{reg}$ the natural inclusion.
Since $\Sigma$ is discrete, the link $L$ can be
written as a disjoint union $L=\sqcup L_\sigma$, where $L_\sigma=L(\sigma,X)$ is the link of $\sigma\in\Sigma$
in $X$.

Assume that $X$ admits a resolution of singularities 
$f:\wt X\lra X$ of $X$ such that the exceptional
divisor $D:=f^{-1}(X)$ is smooth.
Denote by
$$j^k:H^k(\wt X)\lra H^k(D)\text{ and }\gamma^k:H^{k-2}(D)\lra H^k(\wt X)$$
the restriction morphisms and the Gysin maps 
induced by the inclusion $j:D\hookrightarrow \wt X$.
For all $k\geq 2$, define $j_{\#}^k:=j^k\circ \gamma^k:H^{k-2}(D)\lra H^{k}(D)$.

Unless stated otherwise, all cohomologies are taking with rational coefficients.

\subsection{Perverse weight spectral sequence}\label{perverseordinary}
The morphism $E_1(\iota^*):E_1^{*,*}(X_{reg})\lra E_1^{*,*}(L)$ of weight spectral sequences
induced by the inclusion $\iota:L\hookrightarrow X_{reg}$ can be written as:

$$
\xymatrix@R=8pt@C=36pt{
E_1^{r,s}(X_{reg})=\ar[ddd]&\ar[ddd]^{Id} H^{s-2}(D)\ar[r]^{\gamma^s}&\ar[ddd]^{j^s} H^{s}(\wt X)\\
\\
\\
E_1^{r,s}(L)=& H^{s-2}(D)\ar[r]^{j_{\#}^s}& H^{s}(D)\\
&\text{\tiny{$r=-1$}}&\text{\tiny{$r=0$}}&
}
$$
The algebra structure of $E_1^{*,*}(X_{reg})$ is induced by the 
cup product of $H^*(\wt X)$, together with the maps
$H^s(\wt X)\times H^{s'}(D)\lra H^{s+s'}(D)$
given by $(x,a)\mapsto j^*(x)\cdot a$.
Since $\gamma(a\cdot j^*(x))=\gamma(a)\cdot x$,
this algebra structure is compatible with the differential $\gamma$.
The non-trivial products of $E_1^{*,*}(L)$ are the maps $E_1^{0,s}(L)\times E_1^{r,s'}(L)\lra E_1^{r,s+s'}(L)$,
with $r\in\{0,1\}$ and $s,s'\geq 0$, induced by the cup product of $H^*(D)$.

The perverse weight spectral sequence $IE^{*,*}_{1,\ov \bullet}(X):=\Ii_{\ov{\bullet}}(E_1(i^*))$ for $X$ 
 can be written as:
\begin{equation*}
\resizebox{1\hsize}{!}{$
IE^{r,s}_{1,\ov{p}}(X)=
\def\arraystretch{1.6}
\begin{array}{| c c c c c |}
\hline
H^{s-2}(D)\otimes\Lambda(t)\otimes t&\lra&\Jj^s_{1}\oplus H^{s-2}(D)\otimes\Lambda(t)\otimes dt&\lra&H^{s}(D)\otimes\Lambda(t)\otimes dt\\ \hline
\Ker(j_{\#}^s)\oplus H^{s-2}(D)\otimes\Lambda(t)\otimes t&\lra&\Jj^s_{1}\oplus H^{s-2}(D)\otimes\Lambda(t)\otimes dt&\lra&H^{s}(D)\otimes\Lambda(t)\otimes dt\\ \hline
H^{s-2}(D)\otimes\Lambda(t)&\lra&\Jj^s_{0}\oplus H^{s-2}(D)\otimes\Lambda(t)\otimes dt&\lra&H^{s}(D)\otimes\Lambda(t)\otimes dt\\ \hline \hline
\multicolumn{1}{c}{\text{\tiny{$r=-1$}}}&&\multicolumn{1}{c}{\text{\tiny{$r=0$}}}&&\multicolumn{1}{c}{\text{\tiny{$r=1$}}}
\end{array}
\begin{array}{ l }
\text{\tiny{$s>p+1$}}\\
\text{\tiny{$s=p+1$}}\\
\text{\tiny{$s<p+1$}}\\
\multicolumn{1}{c}{\tiny{}}
\end{array}
$}\end{equation*}

where $\Jj^s_{\alpha}$,with $\alpha\in\{0,1\}$ and $0\leq s\leq n$, is the vector space defined via the pull-back:
$$
\xymatrix{
\pb\ar[d]
\Jj^s_{\alpha}\ar[r]&H^{s}(D)\otimes\Lambda(t)\otimes t^{\alpha}\ar[d]^{\delta_1}\\
H^{s}(\wt X)\ar[r]^{j^s}&H^{s}(D)
}.
$$

The differential $d_1^{-1,s}:IE^{-1,s}_{1,\ov{p}}(X)\lra IE^{0,s}_{1,\ov{p}}(X)$ is given by
$$\sum a_it^i\mapsto \left((\sum \gamma^s(a_i), \sum j_{\#}^s(a_i)t^i), \sum i a_i t^{i-1}dt)\right); \, a_i\in H^{s-2}(D).$$
The differential $d_1^{0,s}:IE^{0,s}_{1,\ov{p}}(X)\lra IE^{-1,s}_{0,\ov{p}}(X)$ is given by
$$\left( (x,\sum a_i t^i), \sum b_it^i dt)\right)\mapsto \sum i a_i t^{i-1} dt+\sum j_{\#}^s(b_i)t^i dt\,; 
\left\{\begin{array}{l}
a_i \in H^{s}(D), b_i\in H^{s-2}(D)\\
x\in H^{s}(X), j^s(x)=\sum a_i
\end{array}\right..$$

The algebra structure of $IE^{*,*}_{1,\ov \bullet}(X)$ is given by the following maps:
$$
\def\arraystretch{1.6}
\begin{array}{llll}
\jb\,& IE^{0,*}_{1,\ov p}(X)\times IE^{0,*}_{1,\ov q}(X)&\lra &IE^{0,*}_{1,\ov p+\ov q}(X)\\
& ((x,a+b\cdot dt), (y,c+e\cdot dt))&\mapsto &(x y, a c+(a e+c b)\cdot dt)\\
\jb\,&IE^{0,*}_{1,\ov p}(X)\times IE^{-1,*}_{1,\ov q}(X)&\lra &IE^{-1,*}_{1,\ov p+\ov q}(X)\\ 
 & ((x,a+b\cdot dt), c)&\mapsto &a c\\
\jb\,& IE^{0,*}_{1,\ov p}(X)\times IE^{1,*}_{1,\ov q}(X)&\lra &IE^{1,*}_{1,\ov p+\ov q}(X)\\ 
& ((x,a+b\cdot dt), c\cdot dt)&\mapsto &a c\cdot dt\\
\jb\,&  IE^{1,*}_{1,\ov p}(X)\times IE^{-1,*}_{1,\ov q}(X)&\lra &IE^{0,*}_{1,\ov p+\ov q}(X)\\ 
& (a,c\cdot dt)&\mapsto &a c \cdot dt
\end{array}
$$
where $x,y\in H^*(\wt X)$ and $a,b,c,e\in H^*(D)\otimes\Lambda(t)$.

By computing the cohomology of $IE_{1,\ov \bullet}(X)$ we find:

$$
IE^{r,s}_{2,\ov{p}}(X)\cong
\arraycolsep=18pt\def\arraystretch{1.6}
\begin{array}{| c | c | c |}
\hline
0&\Ker(j^s)&\Coker(j^s)\\ \hline
\Ker(\gamma^s)&\Ker(j^s)/(\Ker(j^s)\cap\Img(\gamma^s))&\Coker(j^s)\\ \hline
\Ker(\gamma^s)&\Coker(\gamma^s)&0\\ \hline \hline
\multicolumn{1}{c}{\text{\tiny{$r=-1$}}}&\multicolumn{1}{c}{\text{\tiny{$r=0$}}}&\multicolumn{1}{c}{\text{\tiny{$r=1$}}}
\end{array}
\begin{array}{ l }
\text{\tiny{$s>p+1$}}\\
\text{\tiny{$s=p+1$}}\\
\text{\tiny{$s<p+1$}}\\
\multicolumn{1}{c}{\tiny{}}
\end{array}
$$

This gives the following formula for the intersection cohomology of $X$:
$$IH^s_{\ov p}(X;\QQ)=\left\{
\begin{array}{ll}
H^s(X_{reg})\cong \Ker(\gamma^{s+1})\oplus  \Coker(\gamma^s) &\text{ if }s<p+1\\
\Img(H^s(\wt X)\lra H^s(X_{reg}))\cong \dfrac{\Ker(j^s)}{\Img(\gamma^s)\cap \Ker(j^s)}&\text{ if }s=p+1\\
H^s(X)\cong \Coker(j^{s-1})\oplus \Ker(j^{s})&\text{ if }s>p+1
\end{array}
\right..
$$

The following is straightforward.
\begin{lem}\label{PD}
For all $0\leq s\leq n$ have  Poincar\'{e} duality isomorphisms
$$\Coker(\gamma^{n+s})\cong \Ker(j^{n-s})^\du\text{ and }\Ker(\gamma^{n+s})\cong \Coker(j^{n-s})^\du.$$
\end{lem}

\subsection{Conditions on the cohomology of the link}
Since $\dim(\Sigma)=0$, the weight filtration on the cohomology of the link
is semi-pure: the weights on $H^k(L)$ are less than or equal to $k$ for $k<n$,
and greater or equal to $k+1$ for $k\geq n$.
This is a consequence of Gabber's purity theorem and the decomposition theorem of
intersection homology (see \cite{Ste}, see also \cite{Na2} for a direct proof).

\begin{lem}\label{semipure0}With the above notation we have:
\begin{enumerate}[(1)]
\item The map $j_{\#}^k:H^{k-2}(D)\to H^k(D)$ is injective for $k\leq n$ and surjective for $k\geq n$.
\item The Gysin map $\gamma^k:H^{k-2}(D)\to H^k(\wt X)$ is injective for $k\leq n$ and $\gamma^{2n}$ is surjective.
\item The restriction morphism $j^k:H^{k}(\wt X)\to H^k(D)$ is surjective for $k\geq n$.
\end{enumerate}
\end{lem}
\begin{proof}
The first assertion follows from the semi-purity of the link and the isomorphisms
$$Gr_{k+1}^WH^{k}(L)\cong \Ker(j_{\#}^{k+1})\text{ and } Gr_{k}^WH^{k}(L)\cong \Coker(j_{\#}^{k}).$$
Since $j_{\#}^k=j^k\circ \gamma^k$, it follows from (1) that $\gamma^k$ is injective for $k\leq n$ and
$j^k$ is surjective for $k\geq n$. Since $H^{2n}(X_{reg})\cong \Coker(\gamma^{2n})$ and $X_{reg}$ is non-compact of real dimension $2n$,
$\gamma^{2n}$ is surjective.
\end{proof}

Assume that the rational cohomology of the link $L$ of $\Sigma$ in $X$ satisfies
$H^i(L;\QQ)=0$ for all $0<i\leq n-2$.
For instance, this is the case when $X$ has only ordinary isolated singularities, as we shall later see.

\begin{lem}\label{ordinary0}With the above assumption:
\begin{enumerate}[(1)]
 \item   The map $j_{\#}^{n-1}:H^{n-3}(D)\to H^{n-1}(D)$ is injective, the map $j_{\#}^{n+1}:H^{n-1}(D)\to H^{n+1}(D)$ is surjective 
 and $j_{\#}^s:H^{s-2}(D)\to H^{s}(D)$ is an isomorphism for all $s\neq 0,n-1,n+1,2n$.
 \item  The map $\gamma^s:H^{s-2}(D)\to H^{s}(X)$ is injective for all $s\neq n+1,2n$, and $j^s:H^{s}(X)\to H^{s}(D)$ is surjective for all $s\neq 0,n-1$.
\end{enumerate}
\end{lem}
\begin{proof}
It follows from the isomorphisms $H^s(L)\cong \Ker(j_{\#}^{s+1})\oplus \Coker (j_{\#}^s)$ together with Lemma $\ref{semipure0}$.
\end{proof}

\begin{lem}\label{ordinary1}With the above assumption we have the following isomorphisms:
\begin{enumerate}[(1)]
 \item  $H^s(X)\cong \Ker(j^s)\oplus \Img(\gamma^s)$ for all $s\neq 0,n-1,n+1,2n$, 
 \item  $\Ker(j^{n-1})\cap \Img(\gamma^{n-1})=\{0\}$ and $\Ker(j^{n+1})/(\Ker(j^{n+1})\cap \Img(\gamma^{n+1}))\cong \Coker(\gamma^{n+1})$.
\end{enumerate}
\end{lem}
\begin{proof}
Assume that $s\neq 0,n-1,n+1,2n$. From (1) and (2) of Lemma $\ref{ordinary0}$, $j_{\#}^s$ is an isomorphism and $j^s$ is surjective.
Then the composition $\gamma^s\circ (j_{\#}^s)^{-1}$ defines a splitting for the short exact sequence
$$0\lra \Ker(j^{s})\lra H^s(X)\stackrel{j^{s}}{\lra}H^s(D)\lra 0$$
Hence (1) follows. The isomorphisms in (2) follow from the injectivity of $j_{\#}^{n-1}$ and the surjectivity of $j_{\#}^{n+1}$ respectively.
\end{proof}

\subsection{Intersection cohomology}
By Lemma $\ref{ordinary1}$ the second term of the weight spectral sequences for $X_{reg}$ and $L$ are given by:
$$
E_2^{r,s}(X_{reg})\cong
\def\arraystretch{1.6}
\begin{tabular}{| c | c | }
\hline
$\Ker(\gamma^{2n})$&0\\ \hline
0&$\Coker(\gamma^s)$\\ \hline
$\Ker(\gamma^{n+1})$&$\Coker(\gamma^{n+1})$\\ \hline
& \\
0&$\Coker(\gamma^s)$\\
& \\\hline
0&$H^0(\wt X)$\\ \hline\hline
\multicolumn{1}{c}{\tiny{$r=-1$}}&\multicolumn{1}{c}{\tiny{$r=0$}}
\end{tabular}
\,\,\,\,\,;\,\,\,\,\,
E_2^{r,s}(L)\cong
\begin{tabular}{| c | c | }
\hline
$H^{2n-2}(D)$&0\\ \hline
0&0\\  \hline
$\Ker(j_{\#}^{n+1})$&0\\ \hline
0&0\\ \hline
0&$\Coker(j_{\#}^{n-1})$\\ \hline
0&0\\ \hline
0&$H^0(D)$\\ \hline \hline
\multicolumn{1}{c}{\tiny{$r=-1$}}&\multicolumn{1}{c}{\tiny{$r=0$}}
\end{tabular}\,\,\,\,
\begin{tabular}{ l }
\tiny{$s=2n$}\\
\tiny{$s>n+1$}\\
\tiny{$s=n+1$}\\
\tiny{$s=n$}\\
\tiny{$s=n-1$}\\
\tiny{$s<n-1$}\\
\tiny{$s=0$}\\ \\
\end{tabular}
$$

By Lemmas $\ref{ordinary0}$ and $\ref{ordinary1}$ the second term $IE^{r,s}_{2,\ov{p}}(X)\cong Gr_{s}^WIH_{\ov p}^{r+s}(X;\QQ)$
of the perverse weight spectral sequence for $X$ is
given in each perversity by the following tables:

$
\def\arraystretch{1.6}
\begin{tabular}{| c | c | c | }
\multicolumn{1}{l}{\tiny{$\ov{0}\leq \ov{p}<\ov{m}$}}\\
\hline
0&$H^{2n}(X)$&0\\ \hline
0&$\Ker(j^s)$&0\\ \hline
0&$\Ker(j^{n-1})$&$\Coker(j^{n-1})$\\ \hline
0&$\Ker(j^{s})$&0\\  \hline
0&$H^0(\wt X)$&0 \\ \hline \hline
\multicolumn{1}{c}{\,\,\,\,\,\,\,\,\,\tiny{$r=-1$}\,\,\,\,\,\,\,\,\,}&
\multicolumn{1}{c}{\,\,\,\,\,\,\,\,\,\,\,\,\,\tiny{$r=0$}\,\,\,\,\,\,\,\,\,\,\,\,\,}&
\multicolumn{1}{c}{\,\,\,\,\,\,\,\,\,\tiny{$r=1$}\,\,\,\,\,\,\,\,\,}
\end{tabular}
\,\,\,;\,\,\,
\begin{tabular}{| c | c | c |}
\multicolumn{1}{l}{\tiny{$\ov{p}=\ov{m}$}}\\
\hline
0&$H^{2n}(X)$&0\\ \hline
0&$\Ker(j^s)$&0\\ \hline
0&$\Coker(\gamma^{n-1})$&0\\ \hline
0&$\Ker(j^{s})$&0\\ \hline
0&$H^0(\wt X)$&0 \\ \hline \hline
\multicolumn{1}{c}{\,\,\,\,\,\,\,\,\,\tiny{$r=-1$}\,\,\,\,\,\,\,\,\,}&
\multicolumn{1}{c}{\,\,\,\,\,\,\,\,\,\,\,\,\tiny{$r=0$}\,\,\,\,\,\,\,\,\,\,\,\,}&
\multicolumn{1}{c}{\,\,\,\,\,\,\,\,\,\tiny{$r=1$}\,\,\,\,\,\,\,\,\,}
\end{tabular}
\,\,\,\,
\begin{tabular}{ l }
\tiny{$s=2n$}\\
\tiny{$s\geq n$}\\
\tiny{$s=n-1$}\\
\tiny{$s<n-1$}\\
\tiny{$s=0$}\\
\end{tabular}
$

$
\arraycolsep=18pt\def\arraystretch{1.6}
\begin{tabular}{| c | c | c |}
\multicolumn{1}{l}{\tiny{$\ov{m}<\ov{p}\leq \ov{t}$}}\\
\hline
0&$H^{2n}(X)$&$0$\\ \hline
0&$\Ker(j^s)$&$0$\\ \hline
$\Ker(\gamma^{n+1})$&$\Coker(\gamma^{n+1})$&$0$\\ \hline
0&$\Ker(j^{n})$&$0$\\ \hline
0&$\Coker(\gamma^{n-1})$&$0$\\ \hline
0&$\Ker(j^{s})$&$0$\\ \hline
0&$H^0(\wt X)$&$0$ \\ \hline \hline
\multicolumn{1}{c}{\,\,\,\,\,\,\,\,\,\tiny{$r=-1$}\,\,\,\,\,\,\,\,\,}&
\multicolumn{1}{c}{\,\,\,\,\,\,\,\,\,\,\,\,\tiny{$r=0$}\,\,\,\,\,\,\,\,\,\,\,\,}&
\multicolumn{1}{c}{\,\,\,\,\,\,\,\,\,\,\,\,\,\tiny{$r=1$}\,\,\,\,\,\,\,\,\,\,\,\,\,}
\end{tabular}
\,\,\,;\,\,\,
\begin{tabular}{| c | c | c |}
\multicolumn{1}{l}{\tiny{$\ov{p}=\ov{\infty}$}}\\
\hline
$\Ker(\gamma^{2n})$&$0$&$0$\\ \hline
0&$\Ker(j^s)$&$0$\\ \hline
$\Ker(\gamma^{n+1})$&$\Coker(\gamma^{n+1})$&$0$\\ \hline
0&$\Ker(j^{n})$&$0$\\ \hline
0&$\Coker(\gamma^{n-1})$&$0$\\ \hline
0&$\Ker(j^{s})$&$0$\\ \hline
0&$H^0(\wt X)$ &$0$\\ \hline \hline
\multicolumn{1}{c}{\,\,\,\,\,\,\,\,\,\tiny{$r=-1$}\,\,\,\,\,\,\,\,\,}&
\multicolumn{1}{c}{\,\,\,\,\,\,\,\,\,\,\,\,\tiny{$r=0$}\,\,\,\,\,\,\,\,\,\,\,\,}&
\multicolumn{1}{c}{\,\,\,\,\,\,\,\,\,\tiny{$r=1$}\,\,\,\,\,\,\,\,\,}
\end{tabular}\,\,\,\,
\begin{tabular}{ l }
\tiny{$s=2n$}\\
\tiny{$s>n+1$}\\
\tiny{$s=n+1$}\\
\tiny{$s=n$}\\
\tiny{$s=n-1$}\\
\tiny{$s<n-1$}\\
\tiny{$s=0$}\\
\end{tabular}
$

\subsection{Intersection-formality}
We now prove the main result of this section.
\begin{teo}\label{intersformal}
Let $X$ be a complex projective variety of dimension $n$ with only isolated singularities. 
Denote by $\Sigma$ the singular locus of $X$.
Assume that there is a resolution of singularities $f:\wt X\lra X$ of $X$ such that $D=f^{-1}(\Sigma)$ is smooth, 
and that the link $L_\sigma$ of $\sigma$ in $X$,
for all $\sigma\in\Sigma$ is $(n-2)$-connected. Then $X$ is GM-intersection-formal over $\CC$. 
Furthermore, if $\Sigma=\{\sigma\}$ is given by a single point then $X$ is intersection-formal over $\CC$.
\end{teo}

\begin{proof}
By Theorem $\ref{IE1formality}$ there is a string of quasi-isomorphisms of 
perverse cdga's from $I\Aa_{\ov\bullet}(X)\otimes\CC$ to
$IE_{1,\ov\bullet}(X)\otimes\CC$. Furthermore, we have $IE_{2,\ov \bullet}(X)\cong IH_{\ov\bullet}(X;\QQ)$.
We next define a perverse cdga $M_{\ov\bullet}$ together with quasi-isomorphisms
$IE_{1,\ov{\bullet}}(X){\longleftarrow} M_{\ov{\bullet}}{\lra} IE_{2,\ov{\bullet}}(X).$

For projective curves, the Theorem is trivially satisfied. Hence we may assume that $n>1$.
We will define $M_{\ov\bullet}$ step by step, for the perversities $\ov 0$, $\ov m$, $\ov n$, $\ov t$ and $\ov \infty$.
We begin with the $\ov 0$-perversity. Let
$M_{\ov{0}}$ be the bigraded complex with trivial differential given by
$$
M_{\ov{0}}^{r,s}=
\arraycolsep=18pt\def\arraystretch{1.6}
\begin{tabular}{| c | c | c | }
\hline
\,\,\,\,\,\,0\,\,\,\,\,\,&$H^{2n}(\wt X)$&0\\ \hline
0&$\Ker(j^s)$&0\\  \hline
0&$\Ker(j^{n-1})$&$\Ker(\gamma^{n+1})^\du\otimes  dt$\\ \hline
0&$\Ker(j^{s})$&0\\ \hline
0&$H^0(\wt X)$&0 \\ \hline \hline
\multicolumn{1}{c}{\tiny{$r=-1$}}&\multicolumn{1}{c}{\tiny{$r=0$}}&\multicolumn{1}{c}{\tiny{$r=1$}}
\end{tabular}
\,\,\,\,
\begin{tabular}{ l }
\tiny{$s=2n$}\\
\tiny{$s\geq n$}\\
\tiny{$s=n-1$}\\
\tiny{$s<n-1$}\\
\tiny{$s=0$}\\\\
\end{tabular}
$$
By Lemma $\ref{PD}$ we have 
$\Ker(\gamma^{n+1})^\du\cong \Coker(j^{n-1})$. The assignation $\Ker(\gamma^{n+1})^\du\otimes  dt\mapsto \Coker(j^{n-1})$
defines an isomorphism of bigraded complexes
$\psi_{\ov 0}:M_{\ov 0}\to IE_{2,\ov 0}(X)$.
We next define an inclusion of bigraded complexes $\varphi_{\ov 0}:M_{\ov 0}\to IE_{1,\ov 0}(X)$.
Recall that
\begin{equation*}\resizebox{1\hsize}{!}{$
IE^{r,s}_{1,\ov{0}}(X)=
\arraycolsep=18pt\def\arraystretch{1.6}
\begin{tabular}{| c | c | c |}
\hline
$H^{s-2}(D)\Lambda(t)\otimes t$&$\left(H^s(\wt X)\oplus_{H^s(D)}H^s(D)\Lambda(t)\otimes t\right)\oplus H^{s-2}(D)\otimes \Lambda(t)\otimes dt$&$H^{s}(D)\otimes \Lambda(t)\otimes dt$\\ \hline
$0$&$H^0(\wt X)\oplus_{H^0(D)}H^0(D)\otimes \Lambda(t)$&$H^{0}(D)\Lambda(t)\otimes dt$\\ \hline \hline
\multicolumn{1}{c}{\tiny{$r=-1$}}&\multicolumn{1}{c}{\tiny{$r=0$}}&\multicolumn{1}{c}{\tiny{$r=1$}}
\end{tabular}
\begin{tabular}{ l }
\tiny{$s\geq 1$}\\
\tiny{$s=0$}\\
\multicolumn{1}{c}{\tiny{}}
\end{tabular}
$}\end{equation*}

For $s\geq 0$ define $M^{0,s}_{\ov{0}}\lra IE_{1,\ov 0}^{0,s}(X)$
by $x\mapsto (x,j^{s}(x))$.
Since $\Ker(\gamma^{n+1})^\du\subset H^{n+1}(D)^\du\cong H^{n-1}(D)$, we have an injective map
$\Ker(\gamma^{n+1})^\du dt \lra IE_{1,\ov{0}}^{1,n-1}(X)=H^{n-1}(D)\otimes\Lambda(t)\otimes dt$.
Then the diagram
$$
\xymatrix@R=5pt{
M^{*,n-1}_{\ov{0}}=&\ar[ddd]\Ker(j^{n-1})\ar[r]^{0}&\Ker(\gamma^{n+1})^\du\otimes dt\ar[ddd]& \\
\\
\\
IE_{1,\ov{0}}^{*,n-1}(X)=&IE_{1,\ov{0}}^{0,n-1}(X)\ar[r]^-{d}&IE_{1,\ov{0}}^{1,n-1}(X)&
}
$$
commutes. Hence we have quasi-isomorphisms of complexes
$IE_{1,\ov{0}}(X)\stackrel{\sim}{\longleftarrow} M_{\ov{0}}\stackrel{\sim}{\lra} IE_{2,\ov{0}}(X).$

The $\ov m$-perversity weight spectral sequence for $X$ is given by:
\begin{equation*}\resizebox{1\hsize}{!}{$
IE^{r,s}_{1,\ov{m}}(X)=
\arraycolsep=18pt\def\arraystretch{1.6}
\begin{tabular}{| c | c | c |}
\hline
$H^{s-2}(D)\otimes \Lambda(t)\otimes t$&$\left(H^s(\wt X)\oplus_{H^s(D)}H^s(D)\otimes \Lambda(t)\otimes t\right)\oplus H^{s-2}(D)\Lambda(t)\otimes dt$&$H^{s}(D)\otimes \Lambda(t)\otimes dt$\\ \hline
$H^{s-2}(D)\otimes \Lambda(t)$&$\left(H^s(\wt X)\oplus_{H^s(D)}H^s(D)\otimes \Lambda(t)\right)\oplus H^{s-2}(D)\otimes \Lambda(t)\otimes dt$&$H^{s}(D)\otimes \Lambda(t)\otimes dt$\\ \hline \hline
\multicolumn{1}{c}{\tiny{$r=-1$}}&\multicolumn{1}{c}{\tiny{$r=0$}}&\multicolumn{1}{c}{\tiny{$r=1$}}
\end{tabular}
\begin{tabular}{ l }
\tiny{$s\geq n$}\\
\tiny{$s<n$}\\
\multicolumn{1}{c}{\tiny{}}
\end{tabular}
$}\end{equation*}

Define $M_{\ov{m}}$ as the bigraded sub-complex of $IE_{1,\ov{m}}(X)$ given by
$$
M^{r,s}_{\ov{m}}=
\arraycolsep=18pt\def\arraystretch{1.6}
\begin{tabular}{| c | c | c |}
\hline
\,\,\,\,\,0\,\,\,\,\,&$H^{2n}(\wt X)$&0\\ \hline
0&$\Ker(j^s)$&0\\  \hline
0&$\Ker(j^{n+1})^\du\oplus \Ker(\gamma^{n+1})^\du \otimes(t-1)$&$\Ker(\gamma^{n+1})^\du\otimes dt$\\ \hline
0&$\Ker(j^{s})$&0\\\hline
0&$H^0(\wt X)$&0 \\ \hline \hline
\multicolumn{1}{c}{\tiny{$r=-1$}}&\multicolumn{1}{c}{\tiny{$r=0$}}&\multicolumn{1}{c}{\tiny{$r=1$}}
\end{tabular}
\,\,\,\,\,
\begin{tabular}{ l }
\tiny{$s=2n$}\\
\tiny{$s\geq n$}\\
\tiny{$s=n-1$}\\
\tiny{$s<n-1$}\\
\tiny{$s=0$}\\\\
\end{tabular}
$$
The only non-trivial differential of $M_{\ov m}$ is given by differentiation with respect to $t$.
Note that by Lemma $\ref{PD}$ we have 
$\Coker(\gamma^{n-1})\cong \Ker(j^{n+1})^\du\subset H^{n+1}(\wt X)\cong H^{n-1}(\wt X).$
The assignations
$\Ker(j^{n+1})^\du\mapsto \Coker(\gamma^{n-1})$
and $\Ker(\gamma^{n+1})^\du\mapsto 0$ give a commutative
diagram
$$
\xymatrix@R=5pt{
M^{*,n-1}_{\ov{m}}=&\ar[ddd]_{}\Ker(j^{n-1})^\du \oplus \Ker(\gamma^{n+1})^\du \otimes(t-1)\ar[r]^-{d}&\Ker(\gamma^{n+1})^\du\otimes dt\ar[ddd]& \\
\\
\\
IE_{2,\ov{m}}^{*,n-1}(X)=&\Coker(\gamma^{n-1})\ar[r]&0&
}.
$$
Hence we have quasi-isomorphisms of complexes
$IE_{1,\ov{m}}(X)\stackrel{\sim}{\longleftarrow} M_{\ov{m}}\stackrel{\sim}{\lra} IE_{2,\ov{m}}(X).$

By Lemma $\ref{ordinary1}$ we have $\Ker(j^{n-1})\cap \Img(\gamma^{n-1})=\{0\}$. Hence
$\Ker(j^{n-1})\subset \Ker(j^{n+1})^\du$ and we have an
injective morphism of complexes $M_{\ov 0}\lra M_{\ov m}$ making the following diagram commute
$$
\xymatrix{
\ar[d]IE_{1,\ov 0}(X) & \ar[d]\ar[l]_-{\sim}M_{\ov 0}\ar[r]^-{\sim} & \ar[d]IE_{2,\ov 0}(X)\\
IE_{1,\ov m}(X) & \ar[l]_-{\sim}M_{\ov m}\ar[r]^-{\sim} & IE_{2,\ov m}(X)
}.
$$

We next study the case $\ov m+\ov 1=\ov n$.
The $\ov n$-perversity weight spectral sequence for $X$ is given by:
\begin{equation*}\resizebox{1\hsize}{!}{$
IE^{r,s}_{1,\ov{n}}(X)=
\arraycolsep=18pt\def\arraystretch{1.6}
\begin{tabular}{| c | c | c |}
\hline
$H^{s-2}(D)\otimes \Lambda(t)\otimes t$&$\left(H^s(\wt X)\oplus_{H^s(D)}H^s(D)\otimes \Lambda(t)\otimes t\right)\oplus H^{s-2}(D)\otimes \Lambda(t)\otimes dt$&$H^{s}(D)\otimes \Lambda(t)\otimes dt$\\ \hline
$\Ker(j_{\#}^{n+1})\oplus H^{s-2}(D)\otimes \Lambda(t)\otimes t$&$\left(H^s(\wt X)\oplus_{H^s(D)}H^s(D)\otimes \Lambda(t)\otimes t\right)\oplus H^{s-2}(D)\otimes \Lambda(t)\otimes dt$&$H^{s}(D)\otimes \Lambda(t)\otimes dt$\\ \hline
$H^{s-2}(D)\otimes \Lambda(t)$&$\left(H^s(\wt X)\oplus_{H^s(D)}H^s(D)\otimes \Lambda(t)\right)\oplus H^{s-2}(D)\otimes \Lambda(t)\otimes dt$&$H^{s}(D)\otimes \Lambda(t)\otimes dt$\\ \hline \hline
\multicolumn{1}{c}{\tiny{$r=-1$}}&\multicolumn{1}{c}{\tiny{$r=0$}}&\multicolumn{1}{c}{\tiny{$r=1$}}
\end{tabular}
\begin{tabular}{ l }
\tiny{$s> n+1$}\\
\tiny{$s=n+1$}\\
\tiny{$s\leq n$}\\
\multicolumn{1}{c}{\tiny{}}
\end{tabular}
$}\end{equation*}

Define
$M_{\ov{n}}$ as the bigraded sub-complex of $IE_{1,\ov{n}}(X)$ given by
$$
M^{r,s}_{\ov{n}}=
\arraycolsep=18pt\def\arraystretch{1.6}
\begin{tabular}{| c | c | c |}
\hline
$H^{2n-2}(D)\otimes t$&$H^{2n}(\wt X)\oplus H^{2n-2}(D)\otimes dt$&0\\ \hline
0&$\Ker(j^s)$&0\\  \hline
$\Ker(j_{\#}^{n+1})$&$\Ker(j^{n+1})$&0\\ \hline
0&$\Ker(j^{n})$&0\\ \hline
0&$\Ker(j^{n+1})^\du\oplus \Ker(\gamma^{n+1})^\du \otimes (t-1)$&$\Ker(\gamma^{n+1})^\du \otimes dt$\\ \hline
0&$\Ker(j^{s})$&0\\ \hline
0&$H^0(\wt X)$&0 \\ \hline \hline
\multicolumn{1}{c}{\tiny{$r=-1$}}&\multicolumn{1}{c}{\tiny{$r=0$}}&\multicolumn{1}{c}{\tiny{$r=1$}}
\end{tabular}
\,\,\,\,\,
\begin{tabular}{ l }
\tiny{$s=2n$}\\
\tiny{$s>n+1$}\\
\tiny{$s=n+1$}\\
\tiny{$s=n$}\\
\tiny{$s=n-1$}\\
\tiny{$s<n-1$}\\
\tiny{$s=0$}\\\\
\end{tabular}
$$

The non-trivial differentials of $M_{\ov n}$ are given by the map
$\Ker(\gamma^{n+1})^\du\otimes  (t-1)\lra  \Ker(\gamma^{n+1})^\du\otimes  dt$ in degree $s=n-1$ defined by
differentiation with respect to $t$,
the map $\gamma^{n+1}:\Ker(j_{\#}^{n+1})\to \Ker(j^{n+1})$ in degree $s=n+1$ and the map
$H^{2n-2}(D)\otimes t\lra H^{2n}(\wt X)\oplus H^{2n-2}(D)\otimes dt$ in degree $s=2n$
defined by $a\cdot t\mapsto (\gamma^{2n}(a),a\cdot dt)$.
To define a morphism of complexes $M_{\ov{n}}\lra IE_{2,\ov{n}}(X)$ it suffices to
define the maps in degrees $s=n+1$ and $s=2n$.

By Lemma $\ref{ordinary1}$ there is a projection map
$$\Ker(j^{n+1}) \twoheadrightarrow \Ker(j^{n+1})/(\Ker(j^{n+1})\cap \Img(\gamma^{n+1}))\cong \Coker(\gamma^{n+1}).$$
Since $\Ker(\gamma^{n+1})\subset \Ker(j_{\#}^{n+1})$ we may find a direct sum decomposition
$\Ker(j_{\#}^{n+1})=\Ker(\gamma^{n+1})\oplus C$. Define a map $\Ker(j_{\#}^{n+1})\lra \Ker(\gamma^{n+1})$ by projection to the first component.
Since $\gamma^{n+1}(\Ker(j_{\#}^{n+1})\subset \Img(\gamma^{n+1})\cap \Ker(j^{n+1})$ we
have a commutative diagram
$$
\xymatrix@R=5pt{
M^{*,n+1}_{\ov{n}}=&\ar[ddd]\Ker(j_{\#}^{n+1})\ar[r]^-{\gamma^{n+1}}&\Ker(j^{n+1})\ar@{->>}[ddd]& \\
\\
\\
IE_{2,\ov{n}}^{*,n+1}(X)=&\Ker(\gamma^{n+1})\ar[r]^{0}&\Coker(\gamma^{n+1})&
}.
$$
In degree $2n$ we have a commutative diagram
$$
\xymatrix@R=5pt{
M^{*,2n}_{\ov{n}}=&\ar@{->>}[ddd]_{}H^{2n-2}(D)\otimes t\ar[r]^-{d}&H^{2n}(\wt X)\oplus H^{2n-2}(D)\otimes dt\ar@{->>}[ddd]^\pi& \\
\\
\\
IE_{2,\ov{n}}^{*,2n}(X)=&0\ar[r]^{0}&H^{2n}(\wt X)&
}
$$
where 
$d(a\cdot t)=(\gamma^{2n}(a), a\cdot dt)$ and 
$\pi(x,a\cdot dt)=\gamma^{2n}(a)-x$.
This gives quasi-isomorphisms of complexes 
$IE_{1,\ov{n}}(X)\stackrel{\sim}{\longleftarrow} M_{\ov{n}}\stackrel{\sim}{\lra} IE_{2,\ov{n}}(X)$
compatible with the inclusion $M_{\ov m}\to M_{\ov n}$.

The $\ov t$-perversity weight spectral sequence for $X$ is given by:
\begin{equation*}\resizebox{1\hsize}{!}{$
IE^{r,s}_{1,\ov{t}}(X)=
\arraycolsep=18pt\def\arraystretch{1.6}
\begin{tabular}{| c | c | c |}
\hline
$H^{2n-2}(D)\otimes \Lambda(t)\otimes t$&$H^{2n}(X)\oplus H^{2n-2}(D)\otimes \Lambda(t)\otimes dt$&$0$\\ \hline
$H^{s-2}(D)\otimes \Lambda(t)$&$\left(H^s(\wt X)\oplus_{H^s(D)}H^s(D)\Lambda(t)\right)\oplus H^{s-2}(D)\otimes \Lambda(t)\otimes dt$&$H^{s}(D)\otimes \Lambda(t)\otimes dt$\\ \hline \hline
\multicolumn{1}{c}{\tiny{$r=-1$}}&\multicolumn{1}{c}{\tiny{$r=0$}}&\multicolumn{1}{c}{\tiny{$r=1$}}
\end{tabular}
\begin{tabular}{ l }
\tiny{$s=2n$}\\
\tiny{$s<2n$}\\
\multicolumn{1}{c}{\tiny{}}
\end{tabular}
$}\end{equation*}

Define
$M_{\ov{t}}$ as the bigraded sub-complex of $IE_{1,\ov{t}}(X)$ given by
$$
M^{r,s}_{\ov{t}}=
\arraycolsep=18pt\def\arraystretch{1.6}
\begin{tabular}{| c | c | c |}
\hline
$H^{2n-2}(D)\otimes t$&$H^{2n}(\wt X)\oplus H^{2n-2}(D)\otimes dt$&0\\ \hline
0&$\Ker(j^s)$&0\\  \hline
$H^{n-1}(D)$&$H^{n+1}(\wt X)$&0\\ \hline
0&$\Ker(j^{n})$&0\\ \hline
0&$\Ker(j^{n+1})^\du\oplus \Ker(\gamma^{n+1})^\du\otimes  (t-1)$&$\Ker(\gamma^{n+1})^\du \otimes dt$\\ \hline
0&$\Ker(j^{s})$&0\\ \hline
0&$H^0(\wt X)$&0 \\ \hline \hline
\multicolumn{1}{c}{\tiny{$r=-1$}}&\multicolumn{1}{c}{\tiny{$r=0$}}&\multicolumn{1}{c}{\tiny{$r=1$}}
\end{tabular}
\,\,\,\,\,
\begin{tabular}{ l }
\tiny{$s=2n$}\\
\tiny{$s>n+1$}\\
\tiny{$s=n+1$}\\
\tiny{$s=n$}\\
\tiny{$s=n-1$}\\
\tiny{$s<n-1$}\\
\tiny{$s=0$}\\\\
\end{tabular}
$$

Note that $M_{\ov n}$ to $M_{\ov t}$ we only replaced the complex $\gamma^{n+1}:\Ker(j_{\#}^{n+1})\to \Ker(j^{n+1})$
by the quasi-isomorphic complex $\gamma^{n+1}:H^{n-1}(D)\to H^{n+1}(\wt X)$, in degree $s=n+1$.
Analogously to the previous case, choose a direct sum decomposition $H^{n-1}(D)\cong \Ker(\gamma^{n+1})\oplus C\oplus C'$ and define a map
$H^{n-1}(D)\to \Ker(\gamma^{n+1})$ by projecting to the first component.
We have quasi-isomorphisms of complexes
$IE_{1,\ov{t}}(X)\stackrel{\sim}{\longleftarrow} M_{\ov{t}}\stackrel{\sim}{\lra} IE_{2,\ov{t}}(X)$
compatible with the inclusion $M_{\ov n}\to M_{\ov t}$.

Lastly, the $\ov \infty$-perversity weight spectral sequence for $X$ is given by:
\begin{equation*}\resizebox{1\hsize}{!}{$
IE^{r,s}_{1,\ov{\infty}}(X)=
\arraycolsep=18pt\def\arraystretch{1.6}
\begin{tabular}{| c | c | c |}
\multicolumn{1}{c}{}\\
\hline
$H^{s-2}(D)\otimes \Lambda(t)$&$\left(H^s(\wt X)\oplus_{H^s(D)}H^s(D)\otimes \Lambda(t)\right)\oplus H^{s-2}(D)\otimes \Lambda(t)\otimes dt$&$H^{s}(D)\otimes \Lambda(t)\otimes dt$\\ \hline \hline
\multicolumn{1}{c}{\tiny{$r=-1$}}&\multicolumn{1}{c}{\tiny{$r=0$}}&\multicolumn{1}{c}{\tiny{$r=1$}}
\end{tabular}
$}\end{equation*}

Define $M_{\ov{\infty}}$ as the bigraded sub-complex of $IE_{1,\ov{\infty}}(X)$ given by
\begin{equation*}\resizebox{1\hsize}{!}{$
M^{r,s}_{\ov{\infty}}=
\arraycolsep=18pt\def\arraystretch{1.6}
\begin{tabular}{| c | c | c |}
\hline
$H^{2n-2}(D)\oplus H^{2n-2}(D)\otimes t$&$H^{2n}(\wt X)\oplus H^{2n-2}(D)\otimes dt$&0\\ \hline
0&$\Ker(j^s)$&0\\  \hline
$H^{n-1}(D)$&$H^{n+1}(\wt X)$&0\\ \hline
0&$\Ker(j^{n})$&0\\ \hline
0&$\Ker(j^{n+1})^\du\oplus \Ker(\gamma^{n+1})^\du \otimes (t-1)$&$\Ker(\gamma^{n+1})^\du\otimes  dt$\\ \hline
0&$\Ker(j^{s})$&0\\ \hline
0&$H^0(\wt X)$&0 \\ \hline \hline
\multicolumn{1}{c}{\tiny{$r=-1$}}&\multicolumn{1}{c}{\tiny{$r=0$}}&\multicolumn{1}{c}{\tiny{$r=1$}}
\end{tabular}
\,\,\,\,\,
\begin{tabular}{ l }
\tiny{$s=2n$}\\
\tiny{$s>n+1$}\\
\tiny{$s=n+1$}\\
\tiny{$s=n$}\\
\tiny{$s=n-1$}\\
\tiny{$s<n-1$}\\
\tiny{$s=0$}\\\\
\end{tabular}
$}\end{equation*}
Note that from $M_{\ov t}$ to $M_{\ov \infty}$ we only added $H^{2n-2}(D)$ in bidegree $(-1,2n)$.
The differential in degree $s=2n$ is given by $(a+b\cdot t)\mapsto (\gamma^{2n}(a)+\gamma^{2n}(b),b\cdot dt)$.
It is straightforward to define quasi-isomorphisms of complexes
$IE_{1,\ov{\infty}}(X)\stackrel{\sim}{\longleftarrow} M_{\ov \infty}\stackrel{\sim}{\lra} IE_{2,\ov{\infty}}(X)$
compatible with the inclusion $M_{\ov t}\lra M_{\ov \infty}$.

Let $M_{\ov{\bullet}}$ be the perverse complex given by
$$
M_{\ov{p}}=\left\{\begin{array}{ll}
M_{\ov{0}}&\text{ if }0\leq p<n-1\\
M_{\ov{m}}&\text{ if }p=n-1\\
M_{\ov{n}}&\text{ if }p=n\\
M_{\ov{t}}&\text{ if }n< p\leq 2n-2\\
M_{\ov{\infty}}&\text{ if }p=\infty\\
\end{array}\right..
$$
We have quasi-isomorphisms of perverse complexes
$$IE_{1,\ov{\bullet}}(X)\stackrel{\varphi_{\ov \bullet}}{\longleftarrow} M_{\ov{\bullet}}\stackrel{\psi_{\ov \bullet}}{\lra} IE_{2,\ov{\bullet}}(X),$$
where $\varphi_{\ov \bullet}$ is injective and $\psi_{\ov \bullet}$ is surjective.

Consider on $M_{\ov\bullet}$ the multiplicative structure induced by the inclusion $\varphi_{\ov\bullet}$.
We next show that $M_{\ov p}\times M_{\ov q}\subset M_{\ov p+ \ov q}$ for all $\ov p,\ov q\in\widehat \Pp$,
the map $\varphi_{\ov\bullet}$ is a quasi-isomorphism of perverse cdga's.

The multiplicative structure of $M_{\ov0}$ is given by the maps $M_{\ov 0}^{0,*}\times M_{\ov 0}^{0,*}\lra M_{\ov 0}^{0,*}$
induced by the cup product of $H^*(\wt X)$
and the map $M_{\ov 0}^{0,*}\times M_{\ov 0}^{1,*}\lra M_{\ov 0}^{1,*}$ given by
$(a,x)\mapsto a\cdot j^*(x)$. Since $M_{\ov 0}^{0,s}=\Ker(j^s)$ for all $s>0$, the only 
non-trivial product in $M_{\ov 0}^{1,*}$ is
$H^{0}(\wt X)\cdot \Ker(\gamma^{n+1})^\du\lra \Ker(\gamma^{n+1})^\du$.
Also, from the multiplicative structure of $IE_{2,\ov 0}(X)$ we have
$\Ker(\gamma^{n+1})^\du\cdot \Ker(\gamma^{n+1})^\du=0.$
This proves that $M_{\ov 0}\times M_{\ov 0}\subseteq M_{\ov 0}$ and that the map $\varphi_{\ov 0}:M_{\ov 0}\lra IE_{1,\ov 0}(X)$ is a morphism of cdga's.
To see that $M_{\ov 0}\times M_{\ov m}\subseteq M_{\ov m}$ it suffices to prove that
$\Ker(j^{n+1})^\du\cdot \Ker(\gamma^{n+1})^\du=0.$
This follows from the algebra structure of $IE_{2,\ov \bullet}(X)$ together with the
corresponding Poincar\'{e} duality isomorphisms.
We now show that $M_{\ov 0}\times M_{\ov n}\subseteq M_{\ov n}$.
Note that for all $s>0$ we have
$\Ker(j^s)\cdot \Ker(j_{\#}^{n+1})=0.$
The remaining inclusions are trivial.
Therefore $M_{\ov\bullet}$ is a perverse cdga and the inclusion $\varphi_{\ov \bullet}:M_{\ov \bullet}\lra IE_{1,\ov \bullet}(X)$ 
is a quasi-isomorphism of perverse cdga's.

Lastly, we show that for
every pair of perversitites $\ov p,\ov q\in\Pp$  such that $\ov{p}+\ov{q}<\ov{\infty}$, the diagram
$$
\xymatrix{
\ar[d]^{\psi_{\ov p}\otimes \psi_{\ov q}}M_{\ov p}\otimes M_{\ov p}\ar[r]&M_{\ov p+\ov q}\ar[d]^{\psi_{\ov p+\ov q}}\\
IE_{2,\ov p}(X)\otimes IE_{2,\ov q}(X)\ar[r]&IE_{2,\ov p+\ov q}(X)
}
$$
commutes.
The only non-trivial cases are when $\ov p=\ov 0$ and $\ov q=\ov n$ or $q=\ov t$.
We show that the diagram
$$
\xymatrix{
\Ker(j_{\#}^{n+1})\ar[d]\times H^0(\wt X)\ar[d]_{\psi_{\ov n}\times \psi_{\ov 0}}\ar[r]^-{\mu}&\Ker(j_{\#}^{n+1})\ar[d]_{\psi_{\ov n}}\\
\Ker(\gamma^{n+1})\times H^0(\wt X)\ar[r]^-{\mu}&\Ker(\gamma^{n+1})
}
$$
commutes, where $\mu(a,x)=a\cdot j^*(x)$. Recall that the morphism $\psi_{\ov n}:\Ker(j_{\#}^{n+1})\lra \Ker(\gamma^{n+1})$ is defined by
taking a direct sum decomposition $\Ker(j_{\#}^{n+1})=\Ker(\gamma^{n+1})\oplus C$ and choosing the projection to the first component.
Let $(a,x)\in \Ker(j_{\#}^{n+1})\times H^0(\wt X)$, and decompose $a=\overline{a}+c$ with $\overline{a}\in \Ker(\gamma^{n+1})$ and $c\in C$.
Then $\mu(a,x)=(\overline{a}+c)\cdot j^*(x)$. Since $\gamma(\overline{a}\cdot j^*(x))=\gamma(\overline{a})\cdot x=0$,
it suffices to show that $c\cdot j^*(x)\in C$. Since $x=1\in H^0(\wt X)$ and $\gamma(c)\neq 0$, it follows that $\gamma(c\cdot j^*(x))=\gamma(c)\cdot x\neq 0$.
Hence  $c\cdot j^*(x)\in C$, and the above diagram commutes.
We next show that the diagram
$$
\xymatrix{
\Ker(j_{\#}^{n+1})\ar[d]\times \Ker(\gamma^{n+1})^\du\ar[d]_{\psi_{\ov n}\times \psi_{\ov 0}}\ar[r]^-{\mu}&H^{2n-2}(D)dt\ar[d]_{\psi_{\ov n}}\\
\Ker(\gamma^{n+1})\times \Coker(j^{n-1})\ar[r]^-{\mu}&H^{2n}(\wt X)
}
$$
commutes.
Let $(a,b)\in \Ker(j_{\#}^{n+1})\times \Ker(\gamma^{n+1})^\du$. Then
$\psi_{\ov t}(\mu(a,b))=\gamma^{2n}(a\cdot b).$
On the other hand we have 
$\mu(\psi_{\ov t}(a),\psi_{\ov 0} b)=\gamma^{2n}(\overline{a}\cdot b),$
where $a=\overline{a}+c$ is a decomposition such that $\ov{a}\in \Ker(\gamma^{n+1})$ and $c\in C$.
Hence to prove that the above diagram
commutes, it suffices to see that $c\cdot b=0$. This follows from the fact that $C\cap \Ker(\gamma^{n+1})=\{0\}$ and $b\in \Ker(\gamma^{n+1})^\du$.
This proves that 
$\psi_{\ov{0}}\cdot \psi_{\ov{n}}=\psi_{\ov n}$.
The same arguments allow us to prove that $\psi_{\ov{0}}\cdot \psi_{\ov{t}}=\psi_{\ov t}$.
Therefore 
$\psi_{\ov \bullet}$ is multiplicative for finite perversities, and 
$X$ is GM-intersection-formal over $\CC$.

Assume now that $X$ has only one isolated singularity. Then $\Ker(\gamma^{2n})=0$ and the diagram
$$
\xymatrix{
\Ker(j_{\#}^{n+1})\ar[d]\times \Ker(j^{n+1})^\du\ar[d]_{\psi_{\ov n}\times \psi_{\ov m}}\ar[r]^-{\mu}&H^{2n-2}(D)dt\ar[d]_{\psi_{\ov \infty}}\\
\Ker(\gamma^{n+1})\times \Coker(j^{n-1})\ar[r]^-{\mu}&\Ker(\gamma^{2n})
}
$$
commutes. This proves that
$\psi_{\ov{p}}\cdot \psi_{\ov{q}}=\psi_{\ov p + \ov q}$ for all $p,q\in\widehat\Pp$. Hence in this case, $X$ is intersection-formal over $\CC$.
\end{proof}

\subsection{Applications}
A singular point $\sigma\in X$ is called \textit{ordinary} if there exists a neighborhood of $\sigma$
isomorphic to an affine cone $C_\sigma$ with vertex $\sigma$, over a
smooth hypersurface $S_\sigma$ of $\CC\PP^n$.
In such case, the link $L_\sigma$ of $\sigma$ in $X$
is a smooth real manifold of dimension $(2n-1)$ which is $(n-2)$-connected
Hence we have:

\begin{cor}
 Let $X$ be a complex projective variety with only ordinary isolated singularities.
 Then $X$ is GM-intersection-formal over $\CC$.
 Furthermore, if $X$ has only one singular point, then $X$ is intersection-formal over $\CC$.
\end{cor}

\begin{example}[Segre cubic]
Let $S$ denote the set of points $(x_0:x_1:x_2:x_3:x_4:x_5)$ of $\CC\PP^5$
satisfying
$x_0+x_1+x_2+x_3+x_4+x_5=0$ and $x_0^3+x_1^3+x_2^3+x_3^3+x_4^3+x_5^3=0$.
This is a normal projective threefold with 10 isolated ordinary singular points,
known as the \textit{Segre cubic}.
A resolution of $S$ is given by the moduli space $f:\ov{\Mm}_{0,6}\lra S$ of stable rational curves with 6 marked points,
and $D:=f^{-1}(\Sigma)=\bigsqcup_{i=1}^{10} \CC\PP^1\times\CC\PP^1$, where $\Sigma=\{\sigma_1,\cdots,\sigma_{10}\}$
denotes the singular locus of $S$.
For each $0\leq i\leq 10$ the link of $\sigma_i$ in $S$ is homeomorphic to a product
of spheres $L_i\simeq S^2\times S^3$. In particular, $L_i$ is simply connected.
Hence $S$ is GM-intersection-formal over $\CC$.
The intersection homotopy type of $S$ is determined
by the perverse graded algebra $IH_{\ov\bullet}^*(S;\QQ)$, which we next describe.
The rational cohomology of $\ov{\Mm}_{0,6}$ is well-known, with non-trivial Betti numbers:
$b_0(\ov{\Mm}_{0,6})=b_6(\ov{\Mm}_{0,6})=1$ and 
$b_2(\ov{\Mm}_{0,6})=b_4(\ov{\Mm}_{0,6})=16$.
Let $j^s:H^s(\ov{\Mm}_{0,6};\QQ)\to H^s(D;\QQ)$ denote the restriction map induced by the inclusion
$j:D\hookrightarrow \ov{\Mm}_{0,6}$, and $\gamma^s:H^{s-2}(D;\QQ)\to H^s(\ov{\Mm}_{0,6};\QQ)$
the corresponding Gysin map.
The rational cohomology of $S$ is:
$$
H^*(S;\QQ)\cong 
\def\arraystretch{1.4}
\begin{tabular}{| c | }
\hline
$\QQ$\\ \hline
0\\ \hline
$\Ker(j^4)\cong \QQ^6$\\ \hline
$\Coker(j^2)\cong \QQ^5$\\ \hline
$\Ker(j^2)\cong \QQ$\\ \hline
0\\  \hline
$\QQ$ \\ \hline
\end{tabular}
$$
Note that for $k\neq 3$, the weight filtration on $H^k(S;\QQ)$ is pure of weight $k$,
while for $k=3$ we have a non-trivial weight filtration, with $Gr^W_3H^3(S;\QQ)\cong \Ker(j^3)=0$ and
$Gr^W_2H^3(S;\QQ)\cong H^3(S;\QQ)\cong \QQ^5.$

Denote by $Van:=\Coker(j^2)\cong \QQ^5$ and let $Exc\cong \QQ^5$ be defined via the direct sum decomposition
$H^2(\ov{\Mm}_{0,6};\QQ)\cong \Ker(j^2)\oplus \Coker(\gamma^2)\oplus Exc$.
The rational intersection cohomology of $S$ is given by:
$$
IH^*_{\ov p}(S;\QQ)\cong 
\def\arraystretch{1.4}
\begin{tabular}{| c | }
\multicolumn{1}{c}{\tiny{$\ov{0}\leq \ov{p}\leq \ov{1}$}}\\
\hline
$\QQ$\\ \hline
0\\ \hline
$H^2(S;\QQ)^\du\oplus Exc^\du$\\ \hline
$Van$\\ \hline
$H^2(S;\QQ)$\\ \hline
0\\  \hline
$\QQ$ \\ \hline 
\end{tabular}
\,\,\,;\,\,\,
\begin{tabular}{| c |}
\multicolumn{1}{c}{\tiny{$\ov{p}=\ov 2$}}\\
\hline
$\QQ$\\ \hline
0\\ \hline
$H^2(S;\QQ)^\du \oplus Exc^\du$\\ \hline
0\\ \hline
$H^2(S;\QQ)\oplus Exc$\\ \hline
0\\  \hline
$\QQ$ \\ \hline 
\end{tabular}
\,\,\,;\,\,\,
\begin{tabular}{| c | }
\multicolumn{1}{c}{\tiny{$\ov{3}\leq \ov{p}\leq \ov{4}$}}\\
\hline
$\QQ$\\ \hline
0\\ \hline
$H^2(S;\QQ)^\du$\\ \hline
$Van^\du$\\ \hline
$H^2(S;\QQ)\oplus Exc$\\ \hline
0\\  \hline
$\QQ$ \\ \hline 
\end{tabular}
$$
Note that the weight filtration of $IH^*_{\ov \bullet}(S;\QQ)$ is non-trivial, since
$Gr^W_2IH^3_{\ov 0}(S;\QQ)\cong Van\neq 0$.

Since $S$ is simply connected, and $IH_{\ov 0}(S;\QQ)\cong H^*(S;\QQ)$,
one may compute the rational homotopy groups $\pi_*(S)\otimes\QQ$ from a minimal model of
$IH_{\ov 0}^*(S;\QQ)$, as done in Example 4.7 of \cite{ChCi1}.
Likewise, a perverse minimal model (in the sense of \cite{CST}) of the perverse cdga $IH_{\ov \bullet}^*(S;\QQ)$ would give
the ``rational intersection homotopy groups'' of $S$.
\end{example}

For a complete intersection $X$ with isolated singularities, the link of each singular point in $X$ is $(n-2)$-connected
(this result is due to Milnor \cite{Milnor} in the case of hypersurfaces
and to Hamm \cite{Hamm} for general complete intersections).
As another direct consequence of Theorem $\ref{intersformal}$ we have:

\begin{cor}
Let $X$ be a complete intersection with singular locus $\Sigma$ of dimension 0.
Assume that there exists a resolution of singularities $f:\wt X\to X$ such that
$D=f^{-1}(\Sigma)$ is smooth. Then $X$ is GM-intersection-formal over $\CC$.
\end{cor}

\section{Isolated Surface Singularities}\label{Section_ISS}
In this last section we prove that isolated surface singularities are GM-intersection-formal over $\CC$.

\subsection{Notation}
Let $X$ be a complex projective surface with only isolated singularities and denote by $\Sigma$ the singular locus of $X$.
Let $f:\wt X\lra X$ be a resolution of singularities of $X$ such that
$D:=f^{-1}(\Sigma)=D_1\cup\cdots\cup D_N$ is a
simple normal crossings divisor.
Let $\wt D:=D^{(1)}=\sqcup_i D_i$ and $Z:=D^{(2)}=\sqcup_{i\neq j} D_i\cap D_j$. Then
$\wt D$ is a disjoint union of smooth projective curves and $Z$ is a finite collection of points.
Denote by $j:\wt D\lra \wt X$ the natural inclusion.
Let $i_1:Z\to \wt D$ be the inclusion defined by $D_i\cap D_j\mapsto D_i$, for every $i<j$,
and $i_2:Z\to \wt D$ be defined by $D_i\cap D_j\mapsto D_j$, for every $i<j$.
Then
$j\circ i_1=j\circ i_2$.

Denote by
$j^*:H^*(\wt X)\to H^*(\wt D)$ and $i_k^*:H^0(\wt D)\to H^s(Z)$ the induced restriction morphisms
and by 
$\gamma:H^{*}(\wt D)\to H^{*+2}(\wt X)$ 
and 
$\eta_k:H^{0}(Z)\to H^{2}(\wt D)$
the corresponding Gysin maps.
Let $i^*:=i_{1}^*-i_{2}^*$ and $\eta:=\eta_{1}-\eta_{2}$ denote the combinatorial restriction and Gysin maps respectively.
We have $i^*\circ j^*=0$ and $\gamma\circ \eta=0$.
Lastly, define $j_{\#}:=j^*\circ \gamma:H^0(\wt D)\lra H^2(\wt D)$.

\subsection{Perverse weight spectral sequence}\label{perversesurface}
The weight spectral sequence for $X_{reg}$ can be written as
$$
E_1^{r,s}(X_{reg})=
\arraycolsep=1pt
\def\arraystretch{1.6}
\begin{array}{|cccccc|}
\hline
H^0(Z)&&\xra{\,\,\,\eta\,\,\,}&H^2(\wt D)&\xra{\,\,\,\gamma^4\,\,\,}&H^4(\wt X)\\ \hline
&&&H^1(\wt D)&\xra{\,\,\,\gamma^3\,\,\,}&H^3(\wt X)\\ \hline
&&&H^0(\wt D)&\xra{\,\,\,\gamma^2\,\,\,}&H^2(\wt X)\\ \hline
&&&&&H^1(\wt X)\\ \hline
&&&&&H^0(\wt X)\\ \hline
\hline
\multicolumn{1}{c}{\text{\tiny{$r=-2$}}}&&&\multicolumn{1}{c}{\text{\tiny{$r=-1$}}}&&\multicolumn{1}{c}{\text{\tiny{$r=0$}}}
\end{array}
\def\arraystretch{1.6}
\begin{array}{c}
\text{ \tiny{$s=4$}}\\
\text{ \tiny{$s=3$}}\\
\text{ \tiny{$s=2$}}\\
\text{ \tiny{$s=1$}}\\
\text{\tiny{$s=0$}}\\\\
\end{array}
$$
Its algebra structure is
induced by the combinatorial restriction morphisms together with the cup products of $H^*(\wt X)$, $H^*(\wt D)$ and $H^0(Z)$.

The multiplicative weight spectral sequence for the link $L=L(\Sigma,X)$ of $\Sigma$ in $X$ can be written as:
\begin{equation*}
\resizebox{1\hsize}{!}{$
E_1^{r,s}(L)\cong
\arraycolsep=1pt
\def\arraystretch{1.6}
\begin{array}{|cccccccc|}
\hline
H^0(Z)\Lambda(l)&&\xra{\,\,\,d^{-2,4}\,\,\,}&H^2(\wt D)\oplus H^0(Z)\Lambda(l)dl&&&&\\ \hline
&&&H^1(\wt D)&&&&\\ \hline
&&&\Pp(i_1^*,i_2^*)\oplus H^0(Z)\Lambda(l)&\xra{\,\,\,d^{-1,2}\,\,\,}&H^2(\wt D)\oplus H^0(Z)\Lambda(l)dl\oplus H^0(Z)\Lambda(l)dl&&\\ \hline
&&&&&H^1(\wt D)&&\\ \hline
&&&&&\Pp(i_1^*,i_2^*)&\xra{\,\,\,d^{0,0}\,\,\,}&H^0(Z)\Lambda(l)dl\\ \hline
\hline
\multicolumn{1}{c}{\text{\tiny{$r=-2$}}}&&&\multicolumn{1}{c}{\text{\tiny{$r=-1$}}}&&\multicolumn{1}{c}{\text{\tiny{$r=0$}}}&&\multicolumn{1}{c}{\text{\tiny{$r=-1$}}}
\end{array}
\def\arraystretch{1.6}
\begin{array}{c}
\text{ \tiny{$s=4\,\,$}}\\
\text{ \tiny{$s=3\,\,$}}\\
\text{ \tiny{$s=2\,\,$}}\\
\text{ \tiny{$s=1\,\,$}}\\
\text{ \tiny{$s=0\,\,$}}\\\\
\end{array}
$}
\end{equation*}
where $\Pp(i_1^*,i_2^*)$ is the vector space given by the pull-back
$$
\xymatrix{
\pb\ar[d]
\Pp(i_1^*,i_2^*)\ar[r]&H^0(Z)\Lambda(l)\ar[d]^{(\delta_0,\delta_1)}\\
H^{0}(\wt D)\ar[r]^-{(i_1^*,i_2^*)}&H^{0}(Z)\times H^0(Z)}
$$
and the differential is given by
$$\left\{
\def\arraystretch{1.6}
\begin{array}{l}
d^{0,0}(a,z(l))= z'(l)dl,\\
d^{-1,2}((a,z(l)),w(l))=(j_{\#}(a)+\eta_1(w(0))-\eta_2(w(1)),z'(l)dl,w'(l)dl),\\
d^{-2,4}(z(l))=(\eta_1(z(0))-\eta_2(z(1)),z'(l)dl).
\end{array}\right.$$

We describe the morphism $\iota^*:E_1(X_{reg})\lra E_1(L)$ induced by the inclusion $\iota:L\hookrightarrow X_{reg}$.
\begin{itemize}
 \item [\jb]  In degree $s=0$, the map $H^0(\wt X)\lra \Pp(i_1^*,i_2^*)$ is given by $x\mapsto (j^*(x),i_1^*\circ j^*(x))=(j^*(x),i_2^*\circ j^*(x))$.
 \item [\jb]  In degree $s=1$ we have the restriction morphism $j^*:H^1(\wt X)\lra H^1(\wt D)$.

 \item [\jb]  In degree $s=2$, the map $H^0(\wt D)\lra \Pp(i_1^*,i_2^*)\oplus H^0(Z)\Lambda(l)$ is given by
$a\mapsto ((a,i_1^*(a)(1-t)+i_2^*(a)t), 0)$ and
$H^2(\wt X)\lra H^2(\wt D)\oplus H^0(Z)\Lambda(l)dl\oplus H^0(Z)\Lambda(l)dl$
is given by $b\mapsto (j^*(b),-i^*(b), 0)$.
 \item [\jb]  In degree $s=3$ we have the identity of $H^1(\wt D)$.
 \item [\jb]  In degree $s=4$ the maps $H^0(Z)\to H^0(Z)\Lambda(l)$
and $H^2(\wt D)\to H^2(\wt D)\oplus H^0(Z)\Lambda(l)dl$
are given by the natural inclusions $z\mapsto z$ and $a\mapsto (a,0)$ respectively. 
\end{itemize}

\subsection{Intersection cohomology}
By semi-purity, the weight filtration on $H^k(X_{reg})$ is pure of weight $k$, for $k=0,1$.
Also, the weights on $H^k(L)$ are less than or equal to $k$ for $k=0,1$,
and greater or equal to $k+1$ for $k\geq 2$. The 
second terms $E^{r,s}_{2}(X_{reg})\cong Gr_s^WH^{r+s}(X_{reg})$
and  $E^{r,s}_{2}(L)\cong Gr_s^WH^{r+s}(L)$
of the weight spectral sequences for $X_{reg}$ and $L$ respectively are:

\begin{equation*}
\resizebox{1\hsize}{!}{$
E_2^{r,s}(X_{reg})\cong
\def\arraystretch{1.6}
\begin{array}{| c | c | c |}
\hline
\Ker(\eta)&\Ker(\gamma^4)/\Img(\eta)&0\\\hline
\gris&\Ker(\gamma^3)&\Coker(\gamma^3)\\\hline
\gris&0&\Coker(\gamma^2)\\\hline
\gris&\gris&H^1(\wt X)\\\hline
\gris&\gris&H^0(\wt X)\\\hline\hline
\multicolumn{1}{c}{\text{\tiny{$r=-2$}}}&\multicolumn{1}{c}{\text{\tiny{$r=-1$}}}&\multicolumn{1}{c}{\text{\tiny{$r=0$}}}
\end{array}\,\,\,;\,\,\,
E_2^{r,s}(L)\cong
\def\arraystretch{1.6}
\begin{array}{| c | c | c | c |}
\hline
\Ker(\eta)&\Coker(\eta)&\gris&\gris\\\hline
\gris&H^1(\wt D)&\gris&\gris\\\hline
\gris&0&0&\gris\\\hline
\gris&\gris&H^1(\wt D)&\gris\\\hline
\gris&\gris&\Ker(i^*)&\Coker(i^*)\\\hline\hline
\multicolumn{1}{c}{\text{\tiny{$r=-2$}}}&\multicolumn{1}{c}{\text{\tiny{$r=-1$}}}&\multicolumn{1}{c}{\text{\tiny{$r=0$}}}&\multicolumn{1}{c}{\text{\tiny{$r=1$}}}
\end{array} 
\def\arraystretch{1.6}
\begin{array}{c}
\text{\tiny{$s=4$}}\\
\text{\tiny{$s=3$}}\\
\text{\tiny{$s=2$}}\\
\text{\tiny{$s=1$}}\\
\text{\tiny{$s=0$}}\\
\\
\end{array} 
$}\end{equation*}

The second term $IE^{r,s}_{2,\ov p}(X)\cong Gr_s^WIH^{r+s}_{\ov p}(X)$ of the perverse weight spectral sequence for $X$ is:
\begin{equation*}
\resizebox{1\hsize}{!}{$
\def\arraystretch{1.6}
\begin{array}{| c | c | c |}
\multicolumn{1}{c}{}&\multicolumn{1}{c}{\text{\tiny{$\ov p=\ov 0$}}}&\multicolumn{1}{c}{}\\
\hline
H^4(\wt X)&\gris&\gris\\\hline
H^3(\wt X)&\gris&\gris\\\hline
\Ker(j^2)&0&\gris\\\hline
\Ker(j^1)&\Coker(j^1)&\gris\\\hline
H^0(\wt X)&0&\Coker(i^*)\\\hline\hline
\multicolumn{1}{c}{\text{\tiny{$r=0$}}}&\multicolumn{1}{c}{\text{\tiny{$r=1$}}}&\multicolumn{1}{c}{\text{\tiny{$r=2$}}}
\end{array}\,\,\,;\,\,\,
\def\arraystretch{1.6}
\begin{array}{| c   |}
\multicolumn{1}{c}{\text{\tiny{$\ov p=\ov 1$}}}\\
\hline
H^4(\wt X)\\\hline
H^3(\wt X)\\\hline
\Ker(j^2)\\\hline
H^1(\wt X)\\\hline
H^0(\wt X)\\\hline\hline
\multicolumn{1}{c}{\text{\tiny{$r=0$}}}
\end{array}
\,\,\,;\,\,\,
\def\arraystretch{1.6}
\begin{array}{| c | c | c |}
\multicolumn{1}{c}{}&\multicolumn{1}{c}{\text{\tiny{$\ov p=\ov 2$}}}&\multicolumn{1}{c}{}\\
\hline
\Ker(\eta)&0&H^4(\wt X)\\\hline
\gris&\Ker(\gamma^3)&\Coker(\gamma^3)\\\hline
\gris&0&\Coker(\gamma^2)\\\hline
\gris&\gris&H^1(\wt X)\\\hline
\gris&\gris&H^0(\wt X)\\\hline\hline
\multicolumn{1}{c}{\text{\tiny{$r=-2$}}}&\multicolumn{1}{c}{\text{\tiny{$r=-1$}}}&\multicolumn{1}{c}{\text{\tiny{$r=0$}}}
\end{array}
\,\,\,;\,\,\,
\def\arraystretch{1.6}
\begin{array}{| c | c | c |}
\multicolumn{1}{c}{}&\multicolumn{1}{c}{\text{\tiny{$\ov p=\ov \infty$}}}&\multicolumn{1}{c}{}\\
\hline
\Ker(\eta)&\Ker(\gamma^4)/\Img(\eta)&0\\\hline
\gris&\Ker(\gamma^3)&\Coker(\gamma^3)\\\hline
\gris&0&\Coker(\gamma^2)\\\hline
\gris&\gris&H^1(\wt X)\\\hline
\gris&\gris&H^0(\wt X)\\\hline\hline
\multicolumn{1}{c}{\text{\tiny{$r=-2$}}}&\multicolumn{1}{c}{\text{\tiny{$r=-1$}}}&\multicolumn{1}{c}{\text{\tiny{$r=0$}}}
\end{array}
\def\arraystretch{1.6}
\begin{array}{c}
\text{\tiny{$s=4$}}\\
\text{\tiny{$s=3$}}\\
\text{\tiny{$s=2$}}\\
\text{\tiny{$s=1$}}\\
\text{\tiny{$s=0$}}\\
\end{array} 
$}
\end{equation*}
Note that except for $IH_{\ov 0}^2(X)$ and $IH_{\ov 2}^2(X)$,
which may have weights $(0,1,2)$ and $(2,3,4)$ respectively,
the weight filtration on $IH_{\ov p}^s(X)$ is pure of weight $s$.

\subsection{Intersection formality} We prove the main result of this section.
\begin{teo}
Let $X$ be a projective surface with only isolated singularities. Then $X$ is GM-intersection-formal over $\CC$.
Furthermore, if $X$ has only one singular point then $X$ is
intersection-formal over $\CC$.
\end{teo}
\begin{proof}
By Theorem $\ref{IE1formality}$ there is a string of quasi-isomorphisms of 
perverse cdga's from $I\Aa_{\ov\bullet}(X)\otimes\CC$ to
$IE_{1,\ov\bullet}(X)\otimes\CC$. Furthermore, we have $IE_{2,\ov \bullet}(X)\cong IH_{\ov\bullet}(X;\QQ)$.
We will next define a perverse cdga $M_{\ov\bullet}$ together with quasi-isomorphisms
$IE_{1,\ov{\bullet}}(X)\stackrel{\sim}{\longleftarrow} M_{\ov{\bullet}}\stackrel{\sim}{\lra} IE_{2,\ov{\bullet}}(X).$

Let $M_{\ov 0}^{*,*}$ be the bigraded complex with trivial differential given by
\begin{equation*}
M_{\ov 0}^{r,s}=
\def\arraystretch{1.6}
\begin{array}{| c | c | c |}
\hline
H^4(\wt X)&\gris&\gris\\\hline
H^3(\wt X)&\gris&\gris\\\hline
\Ker(j^2)&0&\gris\\\hline
\Ker(j^1)&\Ker(\gamma^3)^\du \otimes dt&\gris\\\hline
H^0(\wt X)&0&\Ker(\eta)^\du \otimes  dt\\\hline\hline
\multicolumn{1}{c}{\text{\tiny{$r=0$}}}&\multicolumn{1}{c}{\text{\tiny{$r=1$}}}&\multicolumn{1}{c}{\text{\tiny{$r=2$}}}
\end{array}
\,\,\,\,
\begin{tabular}{ l }
\tiny{$s=4$}\\
\tiny{$s=3$}\\
\tiny{$s=2$}\\
\tiny{$s=1$}\\
\tiny{$s=0$}\\\\
\end{tabular}
\end{equation*}
By Poincar\'{e} duality we have $\Ker(\eta)^\du\cong \Coker(i^*)$ and $\Ker(\gamma^3)^\du\cong \Coker(j^1)$. The assignation $dt\mapsto 1$
defines an isomorphism of bigraded complexes
$\psi_{\ov 0}:M_{\ov 0}\to IE_{2,\ov 0}(X)$.
We next define an injective morphism
$\varphi_{\ov 0}:M_{\ov 0}\lra IE_{1,\ov 0}(X)$. 
To ease notation we will let $E^{r,s}:=E_1^{r,s}(L)$ and $d^{r,s}:=d_1^{r,s}(L)$.
The $\ov 0$-perversity weight spectral sequence $IE_{1,\ov 0}(X)$ for $X$ is:
\begin{equation*}\resizebox{1\hsize}{!}{$
\def\arraystretch{2}
\begin{array}{| c | c | c | c | c |}
\hline
H^0(Z)\wh\oplus E^{-2,4}\Lambda(t)\otimes t&
\left( H^2(\wt D)\wh\oplus E^{-1,4}\Lambda(t)\otimes t\right) \oplus E^{-2,4}\Lambda(t)\otimes dt&
H^4(\wt X)\oplus E^{-1,4}\Lambda(t)\otimes dt&
\gris&\gris\\\hline
\gris&
H^1(\wt D)\wh\oplus E^{-1,3}\Lambda(t)\otimes t&
H^3(\wt X)\oplus E^{-1,3}\Lambda(t)\otimes dt&
\gris&\gris\\\hline
\gris&
H^0(\wt D)\wh\oplus E^{-1,2}\Lambda(t)\otimes t&
\left(H^2(\wt X)\wh\oplus E^{0,2}\Lambda(t)\otimes t\right) \oplus E^{-1,2}\Lambda(t)\otimes dt&
E^{0,2}\Lambda(t)dt&
\gris\\\hline
\gris&\gris&
H^1(\wt X)\wh\oplus E^{0,1}\Lambda(t)\otimes t&
E^{0,1}\Lambda(t)\otimes dt&
\gris\\\hline
\gris&\gris&
H^0(\wt X)\wh\oplus \left(\Ker(d^{0,0})\oplus E^{0,0}\Lambda(t)\otimes t\right)&
E^{1,0}\Lambda(t)\otimes t\otimes (t-1)\oplus E^{0,0}\Lambda(t)\otimes dt&
E^{1,0}\Lambda(t)\otimes dt\\\hline
\multicolumn{1}{c}{r=-2}&\multicolumn{1}{c}{r=-1}&\multicolumn{1}{c}{r=0}&\multicolumn{1}{c}{r=1}&\multicolumn{1}{c}{r=2}\end{array}
$}\end{equation*}

For $s\geq 0$, define an injective morphism $M_{\ov 0}^{0,s}\lra IE_{1,\ov 0}^{0,s}(X)$
by letting $x\mapsto (x,\iota^*(x))$. This is well-defined, since $\iota^*(H^0(\wt X))\subset \Ker(d^{0,0})$
and $\iota^*(\Ker(j^s))=0$. Since $\Ker(\eta)^\du \subset H^0(Z)^\du\cong H^0(Z)$
and $\Ker(\gamma^3)^\du\subset H^1(\wt D)^\du\cong H^1(\wt D)$ we have injective morphisms 
$\Ker(\eta)^\du\lra E^{1,0}\cong H^0(Z)\otimes\Lambda(l)$ and $\Ker(\gamma^3)^\du\lra E^{0,1}\cong H^0(\wt D)$.
We have quasi-isomorphisms
$IE_{1,\ov0}(X)\stackrel{\sim}{\longleftarrow}M_{\ov0}\stackrel{\sim}{\lra}IE_{2,\ov0}(X).$

The middle-perversity $\ov m=\ov 1$ weight spectral sequence $IE_{1,\ov 1}(X)$ for $X$ is:
\begin{equation*}\resizebox{1\hsize}{!}{$
\def\arraystretch{2}
\begin{array}{| c | c | c | c | c |}
\hline
H^0(Z)\wh\oplus E^{-2,4}\Lambda(t)\otimes t&
\left( H^2(\wt D)\wh\oplus E^{-1,4}\Lambda(t)\otimes t\right) \oplus E^{-2,4}\Lambda(t)\otimes dt&
H^4(\wt X)\oplus E^{-1,4}\Lambda(t)\otimes dt&
\gris&\gris\\\hline
\gris&
H^1(\wt D)\wh\oplus E^{-1,3}\Lambda(t)\otimes t&
H^3(\wt X)\oplus E^{-1,3}\Lambda(t)\otimes dt&
\gris&\gris\\\hline
\gris&
H^0(\wt D)\wh\oplus E^{-1,2}\Lambda(t)\otimes t&
\left(H^2(\wt X)\wh\oplus E^{0,2}\Lambda(t)\otimes t\right) \oplus E^{-1,2}\Lambda(t)\otimes dt&
E^{0,2}\Lambda(t)dt&
\gris\\\hline
\gris&\gris&
H^1(\wt X)\wh\oplus E^{0,1}\Lambda(t)&
E^{0,1}\Lambda(t)\otimes dt&
\gris\\\hline
\gris&\gris&
H^0(\wt X)\wh\oplus E^{0,0}\Lambda(t)&
E^{1,0}\Lambda(t)\otimes (t-1)\oplus E^{0,0}\Lambda(t)\otimes dt&
E^{1,0}\Lambda(t)\otimes dt\\\hline
\multicolumn{1}{c}{r=-2}&\multicolumn{1}{c}{r=-1}&\multicolumn{1}{c}{r=0}&\multicolumn{1}{c}{r=1}&\multicolumn{1}{c}{r=2}
\end{array}
$}\end{equation*}

Let $M_{\ov 1}$ be the bigraded sub-complex of $IE_{1,\ov 1}(X)$ given by:
\begin{equation*}
M_{\ov 1}=
\def\arraystretch{1.6}
\begin{array}{| c | c | c |}
\hline
H^4(\wt X)&\gris&\gris\\\hline
H^3(\wt X)&\gris&\gris\\\hline
\Ker(j^2)&0&\gris\\\hline
H^1(\wt X)\oplus \Ker(\gamma^3)^\du\otimes (t-1)&\Ker(\gamma^3)^\du\otimes dt&\gris\\\hline
H^0(\wt X)&\Ker(\eta)^\du\otimes (t-1)&\Ker(\eta)^\du\otimes  dt\\\hline\hline
\multicolumn{1}{c}{\text{\tiny{$r=0$}}}&\multicolumn{1}{c}{\text{\tiny{$r=1$}}}&\multicolumn{1}{c}{\text{\tiny{$r=2$}}}
\end{array}
\end{equation*}
The non-trivial differential of $M_{\ov 1}$ is given by $(t-1)\mapsto dt$.
The assignations $\Ker(\gamma^3)^\du\mapsto 0$ and $\Ker(\eta)^\du\mapsto 0$ define a surjective morphism of complexes $\psi_{\ov 1}:M_{\ov 1}\to IE_{2,\ov 1}(X)$.
Hence we have quasi-isomorphisms
$IE_{1,\ov 1}(X)\stackrel{\sim}{\longleftarrow}M_{\ov1}\stackrel{\sim}{\lra}IE_{2,\ov1}(X)$
compatible with the inclusion $M_{\ov 0}\lra M_{\ov 1}$.

The top-perversity $\ov t=\ov 2$ weight spectral sequence $IE_{1,\ov 2}(X)$ for $X$ is:
\begin{equation*}\resizebox{1\hsize}{!}{$
\def\arraystretch{2}
\begin{array}{| c | c | c | c | c |}
\hline
H^0(Z)\wh\oplus E^{-2,4}\Lambda(t)\otimes t&
\left( H^2(\wt D)\wh\oplus E^{-1,4}\Lambda(t)\otimes t\right) \oplus E^{-2,4}\Lambda(t)\otimes dt&
H^4(\wt X)\oplus E^{-1,4}\Lambda(t)\otimes dt&
\gris&\gris\\\hline
\gris&
H^1(\wt D)\wh\oplus E^{-1,3}\Lambda(t)&
H^3(\wt X)\oplus E^{-1,3}\Lambda(t)\otimes dt&
\gris&\gris\\\hline
\gris&
H^0(\wt D)\wh\oplus E^{-1,2}\Lambda(t)&
\left(H^2(\wt X)\wh\oplus E^{0,2}\Lambda(t)\right) \oplus E^{-1,2}\Lambda(t)\otimes dt&
E^{0,2}\Lambda(t)dt&
\gris\\\hline
\gris&\gris&
H^1(\wt X)\wh\oplus E^{0,1}\Lambda(t)&
E^{0,1}\Lambda(t)\otimes dt&
\gris\\\hline
\gris&\gris&
H^0(\wt X)\wh\oplus E^{0,0}\Lambda(t)&
E^{1,0}\Lambda(t)\otimes (t-1)\oplus E^{0,0}\Lambda(t)\otimes dt&
E^{1,0}\Lambda(t)\otimes dt\\\hline
\multicolumn{1}{c}{r=-2}&\multicolumn{1}{c}{r=-1}&\multicolumn{1}{c}{r=0}&\multicolumn{1}{c}{r=1}&\multicolumn{1}{c}{r=2}
\end{array}
$}\end{equation*}

Let $M_{\ov 2}$ be the bigraded sub-complex of $IE_{1,\ov 2}(X)$ given by:
\begin{equation*}
M_{\ov 2}=
\def\arraystretch{1.6}
\begin{array}{| c | c | c | c | c |}
\hline
\Ker(\eta)^\du&H^2(\wt D)\otimes t&H^4(\wt X)\oplus H^2(\wt D)\otimes dt&\gris&\gris\\\hline
\gris&H^1(\wt D)&H^3(\wt X)&\gris&\gris\\\hline
\gris&0&\Ker(j^2)&0&\gris\\\hline
\gris&\gris&H^1(\wt X)\oplus \Ker(\gamma^3)^\du\otimes (t-1)&\Ker(\gamma^3)^\du\otimes  dt&\gris\\\hline
\,\,\,\,\,\,\,\,\,\,\,\,\gris&\,\,\,\,\,\,\,\,\,\,\,\,\gris&H^0(\wt X)&\Ker(\eta)^\du\otimes (t-1)&\Ker(\eta)^\du\otimes dt\\\hline\hline
\multicolumn{1}{c}{\text{\tiny{$r=-2$}}}&\multicolumn{1}{c}{\text{\tiny{$r=-1$}}}&\multicolumn{1}{c}{\text{\tiny{$r=0$}}}&\multicolumn{1}{c}{\text{\tiny{$r=1$}}}&\multicolumn{1}{c}{\text{\tiny{$r=2$}}}
\end{array}
\end{equation*}
The non-trivial differentials of $M_{\ov 2}$ are given by the map
$H^2(\wt D)\otimes t\lra H^4(\wt X)\oplus H^2(\wt D)\otimes dt$ in degree $s=4$ defined by $a\cdot t\mapsto (\gamma^4(a),a\cdot dt)$,
the Gysin map
$\gamma^3:H^1(\wt D)\to H^3(\wt X)$ in degree $s=3$ and the map $(t-1)\mapsto dt$ defined by differentiation with respect to $t$.
We next define a morphism of complexes $\psi_{\ov 2}:M_{\ov 2}\to IE_{2,\ov 2}(X)$.
In degrees $s=0$ and $s=1$ we let $\psi_{\ov 2}=\psi_{\ov 1}$.
In degree $s=2$ we have a Poincar\'{e} duality isomorphism $\Ker(j^2)\cong \Coker(\gamma^2)$.
For $s=3$ choose a direct sum decomposition $H^1(\wt D)\cong \Ker(\gamma^3)\oplus C_1$
and consider the projection $H^1(\wt D)\to \Ker(\gamma^3)$ to the first component.
This gives a commutative diagram
$$
\xymatrix{
M_{\ov 2}^{*,3}=&H^1(\wt D)\ar[r]^-{\gamma^3}\ar@{->>}[d]&\ar@{->>}[d]H^3(\wt X)&\\
IE_{2,\ov 2}^{*,3}(X)=&\Ker(\gamma^3)\ar[r]^-0&\Coker(\gamma^3)
}
$$
For $s=4$ we have a commutative diagram
$$
\xymatrix{
M_{\ov 2}^{*,4}=&\Ker(\eta)\ar[r]^-{0}\ar[d]^{Id}&\ar[d]H^2(\wt D)\otimes t\ar[r]^-{d}&H^4(\wt X)\oplus H^2(\wt D)\otimes dt\ar@{->>}[d]^\pi\\
IE_{2,\ov 2}^{*,4}(X)=&\Ker(\eta)\ar[r]&0\ar[r]&T
}
$$
where $d(a\cdot t)=(\gamma^4(a),a\cdot dt)$ and $\pi(x,a\cdot dt)=\gamma^4(a)-x$.
Hence we have quasi-isomorphisms of complexes
$IE_{1,\ov 2}(X)\stackrel{\sim}{\longleftarrow}M_{\ov2}\stackrel{\sim}{\lra}IE_{2,\ov2}(X)$
compatible with the inclusion $M_{\ov 1}\lra M_{\ov 2}$.

The $\ov \infty$-perversity weight spectral sequence $IE_{1,\ov \infty}(X)$ for $X$ is:
\begin{equation*}
\resizebox{1\hsize}{!}{$
\def\arraystretch{2}
\begin{array}{| c | c | c | c | c |}
\hline
H^0(Z)\Lambda(t)&
H^2(\wt D)\Lambda(t)\oplus H^0(Z)\Lambda(t)dt&
H^4(\wt X)\oplus H^2(\wt D)\Lambda(t)dt&
\gris&\gris\\\hline
\gris&
H^1(\wt D)\Lambda(t)&
H^3(\wt X)\oplus H^1(\wt D)\Lambda(t)dt&
\gris&\gris\\\hline
\gris&
H^0(\wt D)\Lambda(t)&
H^2(\wt X)\oplus_{H^2(\wt D)}\left(H^2(\wt D)\Lambda(t)\oplus H^0(\wt D)\Lambda(t)dt\right)&
H^2(\wt D)\Lambda(t)dt&
\gris\\\hline
\gris&\gris&
H^1(\wt X)\oplus_{H^1(\wt D)}H^1(\wt D)\Lambda(t)&
H^1(\wt D)\Lambda(t)dt&
\gris\\\hline
\gris&\gris&
H^0(\wt X)\oplus_{H^0(\wt D)}H^0(\wt D)\Lambda(t)t&
H^0(Z)\Lambda(t)(t-1)\oplus H^0(\wt D)\Lambda(t)dt&
H^0(Z)\Lambda(t)dt\\\hline
\multicolumn{1}{c}{\text{\tiny{-2}}}&\multicolumn{1}{c}{\text{\tiny{-1}}}&\multicolumn{1}{c}{\text{\tiny{0}}}&\multicolumn{1}{c}{\text{\tiny{1}}}&\multicolumn{1}{c}{\text{\tiny{2}}}
\end{array}$}
\end{equation*}

Let $M_{\ov \infty}$ be the bigraded sub-complex of $IE_{1,\ov \infty}(X)$ given by:
\begin{equation*}
M_{\ov \infty}=
\def\arraystretch{1.6}
\begin{array}{| c | c | c | c | c |}
\hline
H^0(Z)&H^2(\wt D)\oplus H^2(\wt D)\otimes t&H^4(\wt X)\oplus H^2(\wt D)\otimes dt&\gris&\gris\\\hline
\gris&H^1(\wt D)&H^3(\wt X)&\gris&\gris\\\hline
\gris&0&\Ker(j^2)&0&\gris\\\hline
\gris&\gris&H^1(\wt X)\oplus \Ker(\gamma^3)^\du\otimes (t-1)&\Ker(\gamma^3)^\du\otimes  dt&\gris\\\hline
\,\,\,\,\,\,\,\,\,\,\,\,\gris&\,\,\,\,\,\,\,\,\,\,\,\,\gris&H^0(\wt X)&\Ker(\eta)^\du\otimes (t-1)&\Ker(\eta)^\du\otimes dt\\\hline\hline
\multicolumn{1}{c}{\text{\tiny{$r=-2$}}}&\multicolumn{1}{c}{\text{\tiny{$r=-1$}}}&\multicolumn{1}{c}{\text{\tiny{$r=0$}}}&\multicolumn{1}{c}{\text{\tiny{$r=1$}}}&\multicolumn{1}{c}{\text{\tiny{$r=2$}}}
\end{array}
\end{equation*}
Note that for $s<4$ we have $M_{\ov\infty}^{*,s}=M_{\ov t}^{*,s}$.
In degree $s=4$, the differential of $M_{\ov\infty}$ is given by the map
$H^0(Z)\to H^2(\wt D)\oplus H^2(\wt D)\otimes t$ defined by
$z\mapsto (\eta(z),0)$ and the map $H^2(\wt D)\oplus H^2(\wt D)\otimes t\to H^4(\wt X)\to H^2(\wt D)\otimes dt$
defined by $(a,b\cdot t)\mapsto (\gamma^4(a)+\gamma^4(b),b\cdot dt)$.
To define a surjective morphism of complexes $\psi_{\ov \infty}:M_{\ov \infty}\to IE_{2,\ov \infty}(X)$ it suffices to define
$\psi_{\ov \infty}:M^{*,4}_{\ov \infty}\to IE^{*,4}_{2,\ov \infty}(X)$.
Choose a decomposition $H^0(Z)\cong \Ker(\eta)\oplus C_0$ and
consider the projection $H^0(Z)\to \Ker(\eta)$ to the first component. Also,
choose a decomposition
$H^2(\wt D)\cong \Ker(\gamma^4)\oplus C_2$ and consider the composition
$\rho:H^2(\wt D)\twoheadrightarrow \Ker(\gamma^4)\twoheadrightarrow \Ker(\gamma^4)/\Img(\eta^4)$.
This gives a commutative diagram
$$
\xymatrix{
M_{\ov \infty}^{*,4}=&H^0(Z)\ar[r]^-{d}
\ar@{->>}[d]&\ar[d]^{(\rho,0)}H^2(\wt D)\oplus H^2(\wt D)\otimes t\ar[r]^-{d}&H^4(\wt X)\oplus H^2(\wt D)\otimes dt\ar@{->>}[d]\\
IE_{2,\ov \infty}^{*,4}(X)=&\Ker(\eta)\ar[r]^-0&\Ker(\gamma^4)/\Img(\eta^4)\ar[r]&0
}.
$$
Hence we have
quasi-isomorphisms of complexes
$IE_{1,\ov \infty}(X)\stackrel{\sim}{\longleftarrow}M_{\ov\infty}\stackrel{\sim}{\lra}IE_{2,\ov\infty}(X)$
compatible with the inclusion $M_{\ov 2}\lra M_{\ov \infty}$.

Consider on $M_{\ov\bullet}$ the multiplicative structure induced by the inclusion 
inclusion $\varphi_{\ov\bullet}:M_{\ov\bullet}\to IE_{1,\ov\bullet}(X)$. It is a matter of verification to
see that this structure is closed in $M_{\ov\bullet}$, so that $\varphi_{\ov\bullet}$
 is a morphism of perverse cdga's, which is a quasi-isomorphism.

We next show that for every pair of perversities $\ov p$ and $\ov q$ such that $\ov p+\ov q<\ov \infty$, the diagram
$$
\xymatrix{
M_{\ov p}\times M_{\ov q}\ar[d]_{(\psi_{\ov p}, \psi_{\ov q})}\ar[r]^-{\mu}&M_{\ov p+\ov q}\ar[d]^-{\psi_{\ov p+\ov q}}\\
IE_{2,\ov p}(X)\times IE_{2,\ov q}(X)\ar[r]^-{\ov \mu}&IE_{2,\ov p+\ov q}(X)
}
$$
commutes, so that $\psi_{\ov\bullet}$ is multiplicative for finite perversities. The only non-trivial case is
$$
\xymatrix{
H^1(\wt D)\times \Ker(\gamma^3)^\du\ar[d]_-{(\psi_{\ov 2},\psi_{\ov 0})}\ar[r]^-{\mu}&H^2(\wt D)\ar[d]^-{\psi_{\ov 2}}\\
\Ker(\gamma^3)\times \Ker(\gamma^3)^\du\ar[r]^-{\ov \mu}&H^4(\wt X)
}.
$$
Let $(a,b)\in H^1(\wt D)\times \Ker(\gamma^3)^\du$. Then $\psi_{\ov 2}\mu(a,b)=\gamma^4(a\cdot b)$.
Let $a=\ov a+c$ be a decomposition of $a$ such that $a\in \Ker(\gamma^3)$ and $c\in C_1$.
Then $\ov \mu(\psi_{\ov 2}(a),\psi_{\ov 0}(b))=\gamma^4(\ov a\cdot b)$. Hence to prove that the above diagram commutes it suffices to show that $\gamma^4(c\cdot b)=0$.
This follows from the fact that $\Ker(\gamma^3)\cap C_1=\{0\}$.
This proves that $X$ is GM-intersection-formal over $\CC$.

Assume now that $X$ has only one isolated singularity. Then $\Ker(\gamma^4)/\Img(\eta^4)=0$ and $X$ is
intersection-formal over $\CC$.
\end{proof}

\subsection{An example}
We end with an example of a projective surface with an 
isolated singularity and non-trivial weight filtration on its intersection cohomology.

\begin{example}[Cusp singularity]
Let $C$ be a nodal cubic curve in $\CC\PP^2$. Choose a smooth plane quartic $C'$
intersecting $C$ transversally, so that $|C\cap C'|=12$.
Consider the blow-up $\wt X=Bl_{C\cap C'}\CC\PP^2$ of $\CC\PP^2$ at the $12$ points of $C\cap C'$.
Then the proper transform $\wt C$ of $C$ has negative self-intersection, and we may consider the blow-down $X$ of $\wt C$ to a point. 
Then $X$ is a projective surface with a normal isolated singularity (see $\S7$ of \cite{ToChow},
see also Example 4.2 of \cite{ChCi1} for a more general construction).
To make $\wt C$ into a simple normal crossings divisor
we blow-up $2$ further times at the node of $\wt C$.
This gives a resolution $f:Y\to X$ where $Y\simeq \#_{15} \CC\PP^2$
and the exceptional divisor $D$ is a cycle of three rational curves, so that
$D^{(1)}=\sqcup_{i=1}^3\CC\PP^1$ and $D^{(2)}=\sqcup_{i=1}^3p_i$.
Let $j^s:H^s(Y;\QQ)\lra H^s(D^{(1)};\QQ)$ and $i^*:H^0(D^{(1)};\QQ)\to H^2(D^{(2)};\QQ)$ denote the restriction morphisms.
The rational intersection cohomology of $X$ is given by:
$$
IH^*_{\ov p}(X;\QQ)\cong 
\def\arraystretch{1.4}
\begin{tabular}{| c | }
\multicolumn{1}{c}{\tiny{$\ov{p}=\ov{0}$}}\\
\hline
$\QQ$\\ \hline
0\\ \hline
$\Ker(j^2)\oplus Van$\\ \hline
0\\  \hline
$\QQ$ \\ \hline \hline
\end{tabular}
\,\,\,;\,\,\,
\begin{tabular}{| c |}
\multicolumn{1}{c}{\tiny{$\ov{p}=\ov 1$}}\\
\hline
$\QQ$\\ \hline
0\\ \hline
$\Ker(j^2)$\\ \hline
0\\  \hline
$\QQ$ \\ \hline \hline
\end{tabular}
\,\,\,;\,\,\,
\begin{tabular}{| c | }
\multicolumn{1}{c}{\tiny{$\ov{p}=\ov{2}$}}\\
\hline
$\QQ$\\ \hline
0\\ \hline
$\Ker(j^2)^\vee\oplus Van^\du$\\ \hline
0\\  \hline
$\QQ$ \\ \hline
\end{tabular}
$$
where $Van:=\Coker(i^*)\cong \QQ$ and $\Ker(j^2)\cong \Ker(j^2)^\vee\cong \QQ^{12}$.
The weight filtration on $IH^*_{\ov \bullet}(X;\QQ)$ is non-trivial, with
$Gr^W_2IH^2_{\ov 0}(X;\QQ)\cong \QQ^{12}$, $Gr^W_1IH^2_{\ov 0}(X;\QQ)=0$ and 
$Gr^W_0IH^2_{\ov 0}(X;\QQ)\cong \QQ$.
\end{example}

\bibliographystyle{amsalpha}
\bibliography{bibliografia}

\providecommand{\bysame}{\leavevmode\hbox to3em{\hrulefill}\thinspace}
\providecommand{\MR}{\relax\ifhmode\unskip\space\fi MR }
\providecommand{\MRhref}[2]{%
  \href{http://www.ams.org/mathscinet-getitem?mr=#1}{#2}
}
\providecommand{\href}[2]{#2}
\begin{thebibliography}{DGMS75}

\bibitem[Ban10]{Banagl}
M.~Banagl, \emph{Intersection spaces, spatial homology truncation, and string
  theory}, Lecture Notes in Mathematics, vol. 1997, Springer-Verlag, Berlin,
  2010.

\bibitem[CC]{ChCi1}
D.~Chataur and J.~Cirici, \emph{Rational homotopy of complex projective
  varieties with normal isolated singularities}, Forum. Math. (to appear).

\bibitem[CG14]{CG1}
J.~Cirici and F.~Guill{\'e}n, \emph{{$E\sb 1$}-formality of complex algebraic
  varieties}, Algebr. Geom. Topol. \textbf{14} (2014), no.~5, 3049--3079.

\bibitem[Cir15]{Cirici}
J.~Cirici, \emph{Cofibrant models of diagrams: mixed {H}odge structures in
  rational homotopy}, Trans. Amer. Math. Soc. \textbf{367} (2015), no.~8,
  5935--5970.

\bibitem[CST]{CST}
D.~Chataur, M.~Saralegi, and D.~Tanr\'e, \emph{Intersection cohomology.
  {S}implicial blow-up and rational homotopy}, Mem. Amer. Math. Soc. (to
  appear).

\bibitem[Del71]{DeHII}
P.~Deligne, \emph{Th\'eorie de {H}odge. {II}}, Inst. Hautes \'Etudes Sci. Publ.
  Math. (1971), no.~40, 5--57.

\bibitem[Del74]{DeHIII}
\bysame, \emph{Th\'eorie de {H}odge. {III}}, Inst. Hautes \'Etudes Sci. Publ.
  Math. (1974), no.~44, 5--77.

\bibitem[DGMS75]{DGMS}
P.~Deligne, P.~Griffiths, J.~Morgan, and D.~Sullivan, \emph{Real homotopy
  theory of {K}\"ahler manifolds}, Invent. Math. \textbf{29} (1975), no.~3,
  245--274.

\bibitem[DH88]{DH}
A.~H. Durfee and R.~M. Hain, \emph{Mixed {H}odge structures on the homotopy of
  links}, Math. Ann. \textbf{280} (1988), no.~1, 69--83.

\bibitem[Dim92]{Di}
A.~Dimca, \emph{Singularities and topology of hypersurfaces}, Universitext,
  Springer-Verlag, New York, 1992.

\bibitem[Dur83a]{Durfee}
A.H. Durfee, \emph{Mixed {H}odge structures on punctured neighborhoods}, Duke
  Math. J. \textbf{50} (1983), no.~4, 1017--1040.

\bibitem[Dur83b]{Durfee2}
\bysame, \emph{Neighborhoods of algebraic sets}, Trans. Amer. Math. Soc.
  \textbf{276} (1983), no.~2, 517--530.

\bibitem[FM13]{FMC}
G.~Friedman and J.~E. McClure, \emph{Cup and cap products in intersection
  (co)homology}, Adv. Math. \textbf{240} (2013), 383--426.

\bibitem[Fri09]{Friedman}
G.~Friedman, \emph{On the chain-level intersection pairing for {PL}
  pseudomanifolds}, Homology, Homotopy Appl. \textbf{11} (2009), no.~1,
  261--314.

\bibitem[GM80]{GMP1}
M.~Goresky and R.~MacPherson, \emph{Intersection homology theory}, Topology
  \textbf{19} (1980), no.~2, 135--162.

\bibitem[GM83]{GMP2}
\bysame, \emph{Intersection homology. {II}}, Invent. Math. \textbf{72} (1983),
  no.~1, 77--129.

\bibitem[Gor84]{Goresky}
M.~Goresky, \emph{Intersection homology operations}, Comment. Math. Helv.
  \textbf{59} (1984), no.~3, 485--505.

\bibitem[Ham71]{Hamm}
H.~Hamm, \emph{Lokale topologische {E}igenschaften komplexer {R}\"aume}, Math.
  Ann. \textbf{191} (1971), 235--252.

\bibitem[Hov09]{Hov2}
M.~Hovey, \emph{Intersection homological algebra}, New topological contexts for
  {G}alois theory and algebraic geometry ({BIRS} 2008), Geom. Topol. Monogr.,
  vol.~16, Geom. Topol. Publ., Coventry, 2009, pp.~133--150.

\bibitem[Mil68]{Milnor}
J.~Milnor, \emph{Singular points of complex hypersurfaces}, Annals of
  Mathematics Studies, No. 61, Princeton University Press, Princeton, N.J.;
  University of Tokyo Press, Tokyo, 1968.

\bibitem[Mor78]{Mo}
J.~W. Morgan, \emph{The algebraic topology of smooth algebraic varieties},
  Inst. Hautes \'Etudes Sci. Publ. Math. (1978), no.~48, 137--204.

\bibitem[NA85]{Na2}
V.~Navarro-Aznar, \emph{Sur la th\'eorie de {H}odge des vari\'et\'es
  alg\'ebriques \`a singularit\'es isol\'ees}, Ast\'erisque (1985), no.~130,
  272--307, Differential systems and singularities (Luminy, 1983).

\bibitem[NA87]{Na}
\bysame, \emph{Sur la th\'eorie de {H}odge-{D}eligne}, Invent. Math.
  \textbf{90} (1987), no.~1, 11--76.

\bibitem[PS08]{PS}
C.~Peters and J.~Steenbrink, \emph{Mixed {H}odge structures}, Ergebnisse der
  Mathematik und ihrer Grenzgebiete. 3. Folge. A Series of Modern Surveys in
  Mathematics, vol.~52, Springer-Verlag, Berlin, 2008.

\bibitem[Ste83]{Ste}
J.~H.~M. Steenbrink, \emph{Mixed {H}odge structures associated with isolated
  singularities}, Singularities, {P}art 2 ({A}rcata, {C}alif., 1981), Proc.
  Sympos. Pure Math., vol.~40, Amer. Math. Soc., Providence, RI, 1983,
  pp.~513--536.

\bibitem[Sul77]{Su}
D.~Sullivan, \emph{Infinitesimal computations in topology}, Inst. Hautes
  \'Etudes Sci. Publ. Math. (1977), no.~47, 269--331 (1978).

\bibitem[Tot14]{ToChow}
B.~Totaro, \emph{Chow groups, {C}how cohomology, and linear varieties}, Forum
  Math. Sigma \textbf{2} (2014), e17, 25.

\end{thebibliography}

\end{document}